\documentclass[a4paper,intlimits]{amsart}

\usepackage[all]{paper_diening}

\usepackage{amsmath,amsthm,amssymb}

\usepackage{epsfig}
\usepackage[utf8]{inputenc}
\usepackage[T1]{fontenc}
\usepackage{dsfont}

\usepackage{tikz}
\usepackage{pgfplots}
\pgfplotsset{
  width=.65\linewidth,
  axis background/.style={fill=black!5!white},
  grid style={densely dotted,semithick},
  legend style={
    legend columns=1,
    legend pos=outer north east
  },
  compat=newest 
}
 
\pgfplotscreateplotcyclelist{MyColors}{%
    {red,mark = *,every mark/.append style={solid,scale=0.4,fill=red}},
    {teal,mark = x,every mark/.append style={solid,scale=0.4,fill=teal}},
    {blue,mark = square*,every mark/.append style={solid,scale=0.4,fill=blue}},
{dotted,red,mark = *,every mark/.append style={solid,scale=0.4,fill=red}},
    {dotted,teal,mark = x,every mark/.append style={solid,scale=0.4,fill=teal}},
    {dotted,blue,mark = square*,every mark/.append style={solid,scale=0.4,fill=blue}},
{dash dot,red,mark = *,every mark/.append style={solid,scale=0.4,fill=red}},
    {dash dot,teal,mark = x,every mark/.append style={solid,scale=0.4,fill=teal}},
    {dash dot,blue,mark = square*,every mark/.append style={solid,scale=0.4,fill=blue}},}

\pgfplotscreateplotcyclelist{MyColors4}{%
    {red,mark = *,every mark/.append style={solid,scale=0.4,fill=red}},
    {teal,mark = x,every mark/.append style={solid,scale=0.4,fill=teal}},
    {blue,mark = square*,every mark/.append style={solid,scale=0.4,fill=blue}},
    {green,mark = o,every mark/.append style={solid,scale=0.4,fill=green}},
{dotted,red,mark = *,every mark/.append style={solid,scale=0.4,fill=red}},
    {dotted,teal,mark = x,every mark/.append style={solid,scale=0.4,fill=teal}},
    {dotted,blue,mark = square*,every mark/.append style={solid,scale=0.4,fill=blue}},
    {dotted,green,mark = o,every mark/.append style={solid,scale=0.4,fill=green}},
{dash dot,red,mark = *,every mark/.append style={solid,scale=0.4,fill=red}},
    {dash dot,teal,mark = x,every mark/.append style={solid,scale=0.4,fill=teal}},
    {dash dot,blue,mark = square*,every mark/.append style={solid,scale=0.4,fill=blue}},}

\usepackage{mathrsfs}

\numberwithin{equation}{section}

\newcommand{\Div}{\divergence}
\newcommand{\R}{\mathbb R}

\newcommand{\dd}{\,\mathrm{d}}
\newcommand{\ds}{\dd s}
\newcommand{\dt}{\dd t}
\newcommand{\dx}{\dd x}
\newcommand{\dy}{\dd y}
\newcommand{\dz}{\dd z}
\newcommand{\dsigma}{\dd\sigma}

\providecommand{\I}[1][]{\textrm{I}_{#1}}
\providecommand{\II}[1][]{\textrm{II}_{#1}}
\providecommand{\III}[1][]{\textrm{III}_{#1}}

\providecommand{\seminormtmp}[2]{{#1[{#2}#1]}}
\providecommand{\seminorm}[1]{\seminormtmp{}{#1}}

\providecommand{\PiSZzero}{\Pi_{\mathrm{SZ}}^0}
\providecommand{\PiSZone}{\Pi_{\mathrm{SZ}}^1}

\begin{document}

\title{The parabolic $p$-Laplacian with fractional
  differentiability}

\author{Dominic Breit \and Lars Diening \and Johannes Storn \and
  J\"{o}rn Wichmann}%
\address[D. Breit]{Department of Mathematics, Heriot-Watt University,
  Edinburgh EH14 4AS, UK.}%
\email{d.breit@hw.ac.uk}%
\address[L. Diening, J. Storn,
J. Wichmann]{Department of Mathematics, University of Bielefeld,
  Postfach 10 01 31, 33501 Bielefeld, Germany}%
\email{lars.diening@uni-bielefeld.de}%
\email{jstorn@math.uni-bielefeld.de}%
\email{jwichmann@math.uni-bielefeld.de}%

\begin{abstract}
  We study the parabolic $p$-Laplacian system in a bounded domain.  We
  deduce optimal convergence rates for the space-time discretization
  based on an implicit Euler scheme in time. Our estimates are
  expressed in terms of Nikolskii spaces and therefore cover
  situations when the (gradient of) the solution has only fractional
  derivatives in space and time. The main novelty is that, different
  to all previous results, we do not assume any coupling condition
  between the space and time resolution $h$ and $\tau$. The
  theoretical error analysis is complemented by numerical experiments.
\end{abstract}

\thanks{This research was supported by the DFG through the CRC 1283.}


\maketitle

   \begin{center}
   \textbf{Dedicated to the memory of John W. Barrett}
   \end{center}
\section{Introduction}
\label{sec:introduction}
Let $\Omega\subset\R^n$ with Lipschitz boundary, $n\geq 2$, $N\geq1$,
$T>0$ be finite and assume that $\bff:Q\rightarrow\mathbb{R}^N$ and
$\bfu_0: \Omega \rightarrow \R^N$ are given and let $Q:=I\times\Omega$
with $I:=(0,T)$. We are interested in the parabolic $p$-Laplace system
\begin{alignat}{2}
  \label{eq:heat}
  \begin{aligned}
    \partial_t \bfu
    -\Div\big((\kappa+|\nabla\bfu|)^{p-2}\nabla\bfu\big) &=\bff&\qquad
    &\text{in $Q$,}
    \\
    \bfu&=0& &\text{on $I\times\partial\Omega$,}
    \\
    \bfu(0,\cdot) &=\bfu_0 &\qquad& \text{in $\Omega$.}
  \end{aligned}
\end{alignat}
with $\kappa\geq0$ and $p\in(1,\infty)$. The existence of a unique
weak solution to \eqref{eq:heat} in the function space
\begin{align*}
  C(\overline I;L^2(\Omega))\cap L^p(I;W^{1,p}_0(\Omega))
\end{align*}
can be shown by standard monotonicity arguments under very weak
assumptions on the data. We are concerned with its numerical
approximation by finite elements. For this purpose we choose discrete
subspace~$V_h$ of $W^{1,p}_0(\Omega)$, which consists of piece-wise
polynomials on a quasi-uniform triangulation of mesh
size~$h$. Furthermore, we use an implicit Euler scheme with step
size~$\tau= \frac{T}{M+1}$ for the time discretization. The discrete
solution $\bfu_{m,h}$ is given at time points $t_m=m\tau$,
$m=0,\dots,M$ of the time grid.

Many authors have studied the error of these
discretization, e.g. \cite{Wei1992,BaLi2,EbLi,DER,BarDieNoc20}. A variety of quantities
has been used to express the error and many error estimates have been
deduced under different regularity assumptions on the solution~$\bfu$. It
turned out that the natural quantity to measure the error between the
discrete and continuous solution is
\begin{align}
  \label{eq:natural-error}
  \max_{0\leq m\leq M}\Vert\bfu(t_m)-\bfu_{m,h}\Vert_{L^2(\Omega)}^2+ &\sum_{m=1}^M \, \Vert\bfV(\nabla\bfu(t_m))-\bfV(\nabla\bfu_{m,h})\Vert_{L^2(\Omega)}^2,
\end{align}
where $\bfV(\bfxi)=(\kappa+|\bfxi|)^{\frac{p-2}{2}}\bfxi$.

The term
$\norm{\bfV(\nabla\bfu(t_m))-\bfV(\nabla\bfu_{m,h})}_{L^2(\Omega)}^2$
is natural to problems involving the $p$-Laplacian and captures the
nonlinear character of the equation. It has been introduced
by~\cite{BaLi2} for the numerical analysis of the stationary problem
($p$-Poisson problem)
\begin{alignat}{2}
  \label{eq:ppoisson}
  \begin{aligned}
    -\Div\big((\kappa+|\nabla\bfu|)^{p-2}\nabla\bfu\big) &=\bff&\qquad
    &\text{in $\Omega$,}
    \\
    \bfu&=0& &\text{on $\partial\Omega$}
  \end{aligned}
\end{alignat}
in a slightly different but equivalent form under the
name~\emph{quasi-norm}. Note that if $\kappa=0$~\eqref{eq:ppoisson} is the
Euler-Lagrange equation of the
energy~$$\mathcal{J}(\bfv) := \int_\Omega \Big(\frac 1p \abs{\nabla \bfv}^p
- \bfv \cdot \bff\Big)\dx.$$
It has been
observed in~\cite{DieKre08} that the quantity
$\norm{\bfV(\nabla\bfu)-\bfV(\nabla\bfu_h)}_{L^2(\Omega)}^2$ is
equivalent to the energy
error~$\mathcal{J}(\bfu_h)-\mathcal{J}(\bfu)$.  This explains that the
quantities in~\eqref{eq:natural-error} are the natural way to express
the error.  The
variational approach using~$\mathcal{J}$ has been also used
in~\cite{BelDieKre2012} to prove optimal convergence of the adaptive
finite element method for the $p$-Poisson problem using D\"orfler
marking. It has been shown, starting with the seminal paper by
Barrett and Liu~\cite{BaLi1} and with the subsequent extensions by Ebmeyer and
Liu~\cite{EbLi} and by Diening and \Ruzicka{}~\cite{DR}, that solutions
to~\eqref{eq:ppoisson} satisfy
\begin{align*}
  \Vert\bfV(\nabla\bfu(t_m))-\bfV(\nabla\bfu_{m,h})\Vert_{L^2(\Omega)}
  &\lesssim h\, \norm{\nabla \bfV(\nabla \bfu)}_{L^2(\Omega)}.
\end{align*}
The required regularity~$\nabla \bfV(\nabla \bfu)\in L^2(\Omega)$ for the continuous solutions is well-known for problems involving the $p$-Laplacian. It arises naturally when testing the equation by $\Delta\bfu$. This test can be made rigorous by the method of
difference-quotients under appropriate assumptions on the data (for instance for
convex~$\Omega$ or~$\Omega$ with~$C^{1,\alpha}$-boundary).

In the instationary setting the natural regularity using
difference-quotients in time and space is
\begin{subequations}
  \label{eq:regt-full}
  \begin{alignat}{2}
    \bfV(\nabla\bfu) &\in
    L^2(I;W^{1,2}(\Omega)))\cap
    W^{1,2}(I;L^2(\Omega)),
    \\
    \bfu &\in
    L^\infty(I;W^{1,2}(\Omega))\cap
    \mathcal C^{0,1}(\overline I;L^2(\Omega)).
  \end{alignat}
\end{subequations}
It is well-known that weak solutions to \eqref{eq:heat} enjoy the
properties~\eqref{eq:regt-full} provided the data is
regular enough and~$\Omega$ is either convex or has~$C^{1,\alpha}$
boundary.

The expected optimal convergence result for linear elements under the
regularity assumption~\eqref{eq:regt-full} is
\begin{align}
  \label{eq:error0}
  \begin{aligned}
    \max_{0\leq m\leq
      M}\Vert\bfu(t_m)-\bfu_{m,h}\Vert_{L^2(\Omega)}^2+ &\tau \sum_{m=1}^M
    \,   \Vert\bfV(\nabla\bfu(t_m))-\bfV(\nabla\bfu_{m,h})\Vert_{L^2(\Omega)}^2
    \\
    &\lesssim h^2+ \tau^2.
  \end{aligned}
\end{align}
The analysis of implicit Euler schemes for \eqref{eq:heat} started
with the work of Wei~\cite{Wei1992}, who considered the planar case
for~$p\geq 2$ and obtained sub-optimal estimates for the first part of
the error only. In particular, he showed that
$\smash{\max_{ m } \norm{\bfu(t_m) - \bfu_{m,h}}_{L^2(\Omega)}^2}$ is
of order $\smash{h^{\frac{1}{(p-1)}}} + \tau$ provided
that~$\bfu \in C(\overline{I}, W^{2,p}(\Omega))$. Liu and Barett
derived in~\cite{BaLi2} significantly better estimates for all
$1<p<\infty$ , but still sub-optimal compared
to~\eqref{eq:error0}. Instead of~$h^2 +\tau^2$ in~\eqref{eq:error0}
they obtained $h^{\min \set{p,2}}+ \tau$ for
$\smash{\max_{ m } \norm{\bfu(t_m) - \bfu_{m,h}}_{L^2(\Omega)}^2}$
under strong regularity assumptions of the solution.


The optimal rate~\eqref{eq:error0} has been obtained by Diening,
Ebmeyer and~\Ruzicka{} in~\cite{DER} for piece-wise linear elements
under the assumption $p > \frac{2n}{n+2}$.\footnote{The restriction $p>\frac{2n}{n+2}$
comes from the use of Gelfand triples, which requires
$W^{1,p}(\Omega) \embedding L^2(\Omega)$; but could be avoided.} However, their analysis has
the drawback that there is an unnatural coupling of the time-step
size~$\tau$ and~$h$. In particular, for their optimal convergence
result it is required that
\begin{align}
  \label{eq:cond-EDR}
  h^{\beta(p,n)} &\lesssim \tau,
\end{align}
where $\beta(p,n) = 2 - n(1 - \frac p2)$ if $p \in (\frac{n+2}{2n},2]$ and
$\beta(p,n) = n + \frac{2(2-n)}{p}$ for $p \in [2,\infty)$. The bound
for~$p \leq 2$ has been recently improved in~\cite{BerRuz19arxiv}
to~$\beta(p,n) = \frac{4}{p'}$, where $p'=\frac{p}{p-1}$. Note that
different from the well-known CFL condition this is an upper bound
for~$h$ in terms of~$\tau$. Nevertheless, this artificial condition is
very much undesired. It is well known that for the linear case~$p=2$
such a condition is not needed.
The main contribution of this paper is to remove such artificial
restriction completely and to prove that the error
estimate~\eqref{eq:error0} holds for any choice of $h$ and $\tau$
under the regularity assumption~\eqref{eq:regt-full}.

The reason for the coupling between $h$ and $\tau$ in~\eqref{eq:cond-EDR} is the use
of the Scott-Zhang interpolation~$\PiSZone$ operator~\cite{ScoZha90} in
the numerical analysis. This operator has very nice local
properties which have been used in~\cite{DR} to control the
approximation
error~$\norm{\bfV(\nabla \bfu) - \bfV(\nabla \PiSZone \bfu)}_{L^2(\Omega)}^2$ in terms of~$h^2\norm{\nabla \bfV(\nabla \bfu)}_{L^2(\Omega)}^2$. However, the operator~$\PiSZone$ is not self-adjoint and the
treatment of the new term arising from~$\partial_t \bfu$ in the
instationary setting becomes harder to estimate. To overcome this
problem we rather employ the $L^2$-projection $\Pi_2$ onto the finite
element space $V_h$. This, however, requires a control of the new
term~$\norm{\bfV(\nabla \bfu) - \bfV(\nabla \Pi_2
  \bfu)}_{L^2(\Omega)}^2$. Since~$\Pi_2$ is not a local operator, the latter control is
rather delicate. We are able to overcome the arising problems by the use of
sophisticated decay estimates for the~$L^2$-projection due to Eriksson
and Johnson~\cite{ErissonJohnson1995} and Boman~\cite{Boman}. Our
estimates for~$\norm{\bfV(\nabla \bfu) - \bfV(\nabla \Pi_2 \bfu)}_{L^2(\Omega)}^2$
are summarized in Theorem~\ref{thm:Vstab-L2}.

Our approach turns out to be flexible enough to even accommodate
problems with fractional differentiability. If the data (initial
datum, forcing term or boundary of the domain) is not regular enough,
weak solutions fail to enjoy the properties~\eqref{eq:regt-full}. Consequently, an error estimate of the form
\eqref{eq:error0} cannot be expected. In many cases, however, there is
at least some fractional differentiability available and one has
\begin{subequations}
  \label{eq:regt-full-alpha}
  \begin{alignat}{2}
    \bfV(\nabla\bfu) &\in
    L^2(I;N^{\alpha_x,2}(\Omega))\cap
    N^{\alpha_t,2}(I;L^2(\Omega)),
    \\
    \bfu &\in
    L^\infty(I;N^{\alpha_x,2}(\Omega))\cap
    \mathcal{C}^{0,\alpha_t}(\overline{I};L^2(\Omega)),
  \end{alignat}
\end{subequations}
for some $\alpha_x,\alpha_t\in (0,1]$. Here $N^{\alpha,2}$ denotes the
\Nikolskii{} space with differentiability $\alpha\in(0,1]$, see
Section \ref{sec:2} for details. The corresponding error estimate
under these assumptions for~$\alpha_t> \frac 12$ is
\begin{align}
  \label{eq:error1}
  \begin{aligned}
    \max_{0\leq m\leq M}\Vert\bfu(t_m)-\bfu_{m,h}\Vert_{L^2(\Omega)}^2+ &\tau \sum_{m=1}^M \,  \Vert\bfV(\nabla\bfu(t_m))-\bfV(\nabla\bfu_{m,h})\Vert_{L^2(\Omega)}^2
    \\
    &\lesssim h^{2\alpha_x}+ \tau^{2\alpha_t}.
  \end{aligned}
\end{align}
The condition~$\alpha_t > \frac 12$ is necessary for the point
evaluation of~$\bfV(\nabla \bfu)$ using the
embedding~$N^{\alpha_t,2}(I;L^2(\Omega))\hookrightarrow C(\overline
I;L^2(\Omega))$.  Indeed, such an error estimate has been shown by
Breit and Mensah~\cite{BreMen18pre} in the more general situation of
variable exponents~$p=p(t,x)$ but again under some condition
coupling~$h$ and~$\tau$. In particular, they require that
$h \lesssim \tau^{\frac{1+2\alpha_t}{2 \alpha_x}}$. They also require
a very weak form of the CFL-condition, namely that~$\tau^r \lesssim h$
for some arbitrary, large~$r>0$.

For general $\alpha_t \in (0,1]$ we switch to the following averaged
version of the error estimate
\begin{align}
  \label{eq:error2}
  \begin{aligned}
    \max_{1\leq m\leq M}\Vert \mean{\bfu}_{J_m}-\bfu_{m,h}\Vert_{L^2(\Omega)}^2+ &\sum_{m=1}^M \, \int_{t_{m-1}}^{t_{m+1}} \Vert\bfV(\nabla\bfu(s))-\bfV(\nabla\bfu_{m,h})\Vert_{L^2(\Omega)}^2 \ds
    \\
    &\lesssim  h^{2\alpha_x}+ \tau^{2\alpha_t},
  \end{aligned}
\end{align}
where $\mean{\bfu}_{J_m}$ is a time average of~$\bfu$ over the
intervall $J_m = [t_{m-1},t_{m+1}]$ for $m \geq 1$.  This error
estimate is the main result of this paper under the assumption
\begin{subequations}
  \label{eq:regt-full-alpha2}
  \begin{alignat}{2}
    \bfV(\nabla\bfu) &\in
    L^2(I;N^{\alpha_x,2}(\Omega))\cap
    N^{\alpha_t,2}(I;L^2(\Omega)),
    \\
    \bfu &\in L^\infty(I;N^{\alpha_x,2}(\Omega)).
  \end{alignat}
\end{subequations}
All exponents $\alpha_x,\alpha_t\in (0,1]$ are allowed and a coupling
between $h$ and $\tau$ is not needed. The precise statement can be
found in Theorem \ref{thm:main}. If additionally
$\bfu \in \mathcal{C}^{0,\alpha_t}(\overline{I};L^2(\Omega))$, then we
have also control on the pointwise error (see
Remark~\ref{rem:u-hoelder})
\begin{align}
  \max_{1\leq m\leq M}\Vert \bfu(t_m)-\bfu_{m,h}\Vert_{L^2(\Omega)}^2
  &\lesssim  h^{2\alpha_x}+ \tau^{2\alpha_t}.
\end{align}
A main motivation for considering a low time-regularity of the
solution comes from stochastic PDEs. In this case the equations are
driven by a Wiener process which only belongs to the class
$\mathcal C^{0,\alpha_t}(I)$ for all $\alpha_t<\frac{1}{2}$.
Consequently, only a regularity of the form \eqref{eq:regt-full-alpha}
with $\alpha_t<\frac{1}{2}$ is expected. Thus, point evaluations in
time like~$\bfV(\nabla \bfu(t_m))$ as they appear in~\eqref{eq:error1} may
not be possible in the stochastic case.  This problem was circumvented
in~\cite{BreHof}  by the use of randomly perturbed time grids. In
expectation this corresponds to the time averages that we use in this paper.


The paper is organised as follows. In Section \ref{sec:2} we introduce
the analytical setup for equation~\eqref{eq:heat} followed by the
discrete version in Section \ref{sec:discrete-p-heat}. Section
\ref{sec:proj-p-lapl} is devoted to the study of the $L^2$-projection
$\Pi_2$ with respect to the approximability of~$\bfV(\nabla \bfu)$. The
result can be found in Theorem~\ref{thm:Vstab-L2}.  The main error
analysis and the prove of the main result~\eqref{eq:error2} without
any $h$ and $\tau$ coupling can be found in Section~\ref{sec:error} in
Theorem~\ref{thm:main}. Section \ref{sec:exp} contains the results of a numerical simulation study
concerning the discretisation error.
In the appendix we recall some well-known results on Orlicz functions which are needed
throughout the paper.

\section{The continuous equation}
\label{sec:2}

In this section we introduce the analytical setup for
equation~\eqref{eq:heat} including the function spaces.  Let
$\Omega\subset\Rn$ for $n \geq 2$ be a bounded Lipschitz domain
(further assumptions on $\Omega$ will be needed for the regularity of
solutions and the numerical analysis respectively).  For some given
$T>0$ we denote by $I=(0,T)$ the time interval and write
$Q := I \times \Omega$ for the space time cylinder. We write
$f\lesssim g$ for two non-negative quantities $f$ and $g$ if we $f$ is
bounded by $g$ up to a multiplicative constant. The relations
$\gtrsim$ and $\eqsim$
are defined accordingly. We denote by~$c$ a generic constant which can
change its value from line to line.

As usual $L^q(\Omega)$ denotes the Lebesgue spaces and
$W^{1,q}(\Omega)$ the Sobolev spaces, where $1\leq q\leq \infty$. We
denote by~$W^{1,q}_0(\Omega)$ Sobolev spaces with zero boundary
values. It is the closure $C^\infty_0(\Omega)$ (smooth functions with
compact support) in $W^{1,q}(\Omega)$. We denote by
$W^{-1,q'}(\Omega)$ the dual of $W^{1,q}_0(\Omega)$.  In order to
express higher regularity of the solutions we need the notation of
\Nikolskii{} spaces.  For $q\in [1,\infty)$ and $\alpha\in (0,1]$ we
define the semi-norm and norm
\begin{align*}
  \seminorm{u}_{N^{\alpha,q}(\Omega)} &:=
  \sup_{h \in \R^n\backslash\set{0}}
  \abs{h}^{-\alpha}\bigg( \int_{\Omega \cap (\Omega
-h)}|u(x+h)-u(x)|^q \dx
  \bigg)^\frac{1}{q},
  \\
  \norm{u}_{N^{\alpha,q}(\Omega)} &:=
                                    \|u\|_{L^q(\Omega)}+\seminorm{u}_{N^{\alpha,q}(\Omega)}.
\end{align*}
The Nikolskii space $N^{\alpha,q}(\Omega)$ is now defined as the
subspace of $L^q(\Omega)$ consisting of functions having finite
$\norm{\cdot}_{N^{\alpha,q}(\Omega)}$-norm. We call
$\seminorm{\cdot}_{N^{\alpha,q}(\Omega)}$ the semi-norm of
$N^{\alpha,q}(\Omega)$. Vector- and matrix-valued functions will
usually be denoted in bold case, whereas normal case will be adopted
for real-valued functions. We do not distinguish in the notation for
the function spaces.

For a separable Banach space~$(X,\|\cdot\|_X)$ let $L^q(I;X)$ be the
Bochner space of (Bochner-) measureable functions $u:I\rightarrow X$
satisfying $t\mapsto\|u(t)\|_{X}\in L^q(I)$.  Moreover,
$C(\overline{I};X)$ is the space of function
$u:\overline I\rightarrow X$ which are continuous with respect to the
norm-topology. We also use $\mathcal C^{0,\alpha}$ for the space of H\"older
continuous functions and its generalization $C^{k,\alpha}$ for higher
order derivatives. Similarly to the above, we can define fractional
derivatives in time for functions $u:I\rightarrow X$, where
$(X,\|\cdot\|_X)$ is a separable Banach space.  We define for
$q\in[1,\infty)$ and $\alpha\in (0,1]$ the semi-norm and norm
\begin{align*}
  \seminorm{u}_{N^{\alpha,q}(I;X)} &:= \sup_{\tau \in I} \abs{\tau}^{-\alpha}
                                     \bigg(\int_{I \cap (I -
                                     \tau)}\|u(\sigma +\tau)-u(\sigma)\|_{X}^q  \dsigma
                                     \bigg)^{\frac 1q},
  \\
  \norm{u}_{N^{\alpha,q}(I;X)} &:=
                                 \|u\|_{L^{p}(I;X)}+   \seminorm{u}_{N^{\alpha,q}(I;X)}.
\end{align*}
The Nikolskii space $N^{\alpha,q}(I;X)$ is now defined as the subspace
of the Bochner space $L^{q}(I;X)$ consisting of the functions
having finite $\norm{\cdot}_{N^{\alpha,q}(I;X)}$-norm.

For a given force $\bff:Q\rightarrow\mathbb{R}^N$ and initial value
$\bfu_0: \Omega \rightarrow \R^N$ we are interested in the parabolic
$p$-Laplace system
\begin{alignat}{2}
  \label{eq:heat2}
  \begin{aligned}
    \partial_t \bfu
    -\Div \big(\bfS(\nabla \bfu)\big)&=\bff&\qquad
    &\text{in $Q$,}
    \\
    \bfu&=0& &\text{on $I\times\partial\Omega$,}
    \\
    \bfu(0,\cdot) &=\bfu_0 &\qquad& \text{in $\Omega$}
  \end{aligned}
\end{alignat}
with $\kappa\geq0$ and $p\in(1,\infty)$, where
\begin{align}
  \label{eq:def-S}
  \bfS(\nabla \bfu) := (\kappa+|\nabla\bfu|)^{p-2}\nabla\bfu.
\end{align}
We will also later need 
\begin{align} \label{eq:def-V}
\bfV(\nabla \bfu) := (\kappa+|\nabla\bfu|)^\frac{p-2}{2}\nabla\bfu.
\end{align}
It is easy to see that both $\bfS$ and $\bfV$ are monotone and invertible.

As usual we use the following notion of weak solutions.
\begin{definition}
  \label{def:weak}
  Assume that $\bff \in L^1(Q)$ and
  $\bfu_0 \in L^1(\Omega) $.  We call
  $$\bfu\in C(\overline{I};L^2(\Omega))\cap
  L^p(I;W^{1,p}_0(\Omega)) $$ a weak solution to
  \eqref{eq:heat2} if
  \begin{align}
    \label{eq:heatweak1}
    \int_\Omega\bfu(t)\cdot\bfxi\dx
    - \int_\Omega
      \bfu_0\cdot\bfxi\dx+\int_0^t\!\!\!\int_\Omega
      \bfS(\nabla\bfu):\nabla\bfxi\dx \dsigma = \int_0^t\!\!\!\int_\Omega \bff
      \cdot \bfxi \dx \dsigma
  \end{align}
  for all $\bfxi\in C_0^\infty(\Omega)$ and all $t\in I$. 
\end{definition}

The formulation in \eqref{eq:heatweak1} is equivalent to 
\begin{equation}
  \begin{aligned}
    \label{eq:heatweak1''}
    - \int_Q\bfu\cdot\partial_t\bfxi\dx \dt - \int_\Omega \bfu_0\cdot\bfxi(0)\dx +\int_Q
    \bfS(\nabla\bfu):\nabla\bfxi\dx \dt
    = \int_Q \bff\cdot\bfxi\dx
    \dt
  \end{aligned}
\end{equation}
for all $\bfxi\in C_0^\infty([0,T)\times\Omega)$.
It is well known that a weak solution exists provided $\bff \in L^{p'}(I;W^{-1,p'}_0(\Omega))$
and  $\bfu_0 \in  L^2(\Omega) $.

The following regularity result is a special case of \cite[Thm. 4.1]{BreMen18pre} (note that the second inclusion in \eqref{eq:quasispace2} is not explicitly stated in \cite{BreMen18pre} but follows directly from the proof).

\begin{theorem}
  \label{thm:quasispace}
  Let  $\alpha_x,\alpha_t \in (0,1]$ be given and let $\Omega$ be a bounded $C^{1,\alpha_x}$-domain. 
  Let $\bfu$ be the unique weak solution to \eqref{eq:heat} in the sense of Definition \ref{def:weak} with 
  \begin{align}\label{ass:f}
    \bff\in L^{p'}(I;N^{\alpha_x,p'}(\Omega))\cap N^{\alpha_t,2}\big(I;L^2(\Omega)),\\ \bfu_0\in N^{\alpha_x,2}(\Omega),\quad\Div\bfS(\nabla\bfu_0)\in L^2(\Omega).\label{ass:u0}
  \end{align}
  Then we have
  \begin{align}
    \label{eq:quasispace1}
    \bfV(\nabla\bfu)&\in
                      L^2(I;N^{\alpha_x,2}(\Omega))\cap
                      N^{\alpha_t,2}\big(I;L^2(\Omega)),
    \\
    \bfu&\in
          L^\infty(I;N^{\alpha_x,2}(\Omega))\cap
          \mathcal C^{0,\alpha_t}(\overline I;L^2(\Omega)).
          \label{eq:quasispace2}
  \end{align}
\end{theorem}
In the case $\alpha_x=\alpha_t=1$ the result from
Theorem~\ref{thm:quasispace} is classical and follows formally by
testing the equations with $\Delta\bfu$ and $\partial_t^2\bfu$ (see,
for instance, \cite{DSS} for an easy proof). See also~\cite{CiaMaz20}
for sharp, regularity results from testing with the $p$-Laplacian.
Results in a similar spirit concerning the fractional differentiability of nonlinear parabolic systems can be found in \cite{DM} and \cite{DMS}. Results concerning the fractional differentiability of related elliptic problems can be found in \cite{AKM}, \cite{DDHSW}, \cite{EF} and \cite{Sa}.

\section{The discrete equation}
\label{sec:discrete-p-heat}

From now on let $\Omega \subset \setR^n$ be a polyhedral domain.  By
$\mathcal{T}_h$ denote a regular partition (triangulation)
of~$\Omega$ (no hanging nodes), which consists of closed $n$-simplices called
\emph{elements}. For each element ($n$-simplex) $T\in \mathcal{T}_h$
we denote by $h_T$ the diameter of $T$, and by $\rho_T$ the supremum
of the diameters of inscribed balls. By $\abs{T}$ we denote the
Lebesgue measure of~$T$. By $\dashint_T g\dx$ we denote the mean value
integral over the set~$T$. We also abbreviate $\mean{g}_T = \dashint_T
g\dx$ for the mean value.

We assume that $\mathcal{T}_h$ is \emph{shape regular}, that is there exists
a constant~$\gamma$ (the shape regularity constant or chunkiness constant) such that
\begin{align}
  \label{eq:nondeg}
  \max_{T \in \mathcal{T}_h} \frac{h_T}{\rho_T} \leq \gamma.
\end{align}
We define the maximal mesh-size by
\begin{align*}
  h &= \max_{T \in \mathcal{T}_h} h_T.
\end{align*}
We assume further that our triangulation is \emph{quasi-uniform}, i.e.
\begin{align}
  \label{eq:quasi-uniform}
  h_T \eqsim h \qquad \text{for all $T \in \mathcal{T}_h$}.
\end{align}
For $T \in \mathcal{T}_h$ we define the set of neighbors $\omega_T$,
which consists of all elements~$T' \in \mathcal{T}_h$ with $T \cap T'
\neq \emptyset$. We define
\begin{align*}
  \Omega(\omega_T) &:= \Big(\bigcup \set{T\,:\, T\in
                     \omega_T}\Big)^\circ.
\end{align*}
We also assume that~$\Omega(\omega_T)$ is a connected domain for
each~$T$. This only excludes some strange triangulations and is only a
small technicality.

To simplify notations we will use ambiguously $\omega_T$ instead
of~$\Omega(\omega_T)$ for the domain for integrals.

It is easy to see that the shape regularity of~$\mathcal{T}_h$ implies
the following properties, where the constants are independent of $h$:
\begin{enumerate}
\item \label{mesh:SK} $\abs{\omega_T} \eqsim \abs{T}$ for all
  $T \in \mathcal{T}_h$.
\item \label{mesh:NK} There exists $m_1 \in \setN$ such that
  $\# \omega_T \leq m_1$ for all $T \in \mathcal{T}_h$.
\end{enumerate}

For $\ell\in \setN _0$ we denote by $\mathscr{P}_\ell(\Omega)$ the
polynomials on $\Omega$ of degree less than or equal to
$\ell$. Moreover, we set $\mathscr{P}_{-1}(\Omega):=\set {0}$.

For fixed $r \in \setN$ we define the the finite element space $V_h$
as
\begin{align}\label{def:Vh}
  \begin{aligned}
    V_h &:= \set{v \in (W^{1,1}_0(\Omega))^N\,:\, v|_T \in (\mathscr{P}_r(T))^N
      \,\,\forall T\in \mathcal{T}_h}.
  \end{aligned}
\end{align}
\begin{remark}
  For the numerical analysis it is only important that~$V_h$ contains
  all continuous, locally linear functions, and that the functions are
  locally polynomials of a fixed maximal degree. Thus it would for
  example also be possible to use velocity spaces like the
  MINI-element (locally linear functions enriched by bubble
  functions. This might be of interest if our results should be
  applied to the corresponding fluid system.
\end{remark}
  
Let $\{0=t_0<\cdots<t_M=T\}$ be a uniform partition of $[0,T]$ with
mesh size $\tau =T/M$. For $m \geq 1$ define $I_m := [t_{m-1},t_m]$
and $J_m := [t_{m-1},t_{m+1}]$.

For a discrete sequence~$a_m$ we define the backwards-in-time discrete
time derivative~$d_t$ by
\begin{align*}
  d_t a_m := \tau^{-1}(a_m - a_{m-1}).
\end{align*}
Then
\begin{align}
  \label{eq:dtaa}
  d_t a_m \cdot a_m &= \tfrac 12 d_t \abs{a_m}^2 + \tfrac{\tau}{2}
                      \abs{d_t a_m}^2.
\end{align}
Let $\bfu_{0,h}:= \Pi_2\bfu_0$, where $\Pi_2$ is the $L^2$-projection
to~$V_h$. Now for some given time-discrete force $\bff_m$, we define
$\bfu_{m,h}$ as the solution of the implicit Euler scheme
\begin{align}
  \label{eq:tdiscr-pre}
  d_t \bfu_{m,h} - \divergence \big(\bfS(\nabla\bfu_{m,h})\big)
  &=
    \bff_m
    \qquad
    \text{in $V_h^*$}
\end{align}
in the discrete weak sense, i.e.\ for all $\bfxi_h\in V_h$ and
$m =1,\dots, M$ it holds
\begin{align}
  \label{eq:tdiscr}
  \int_\Omega d_t \bfu_{m,h}\cdot\bfxi_h \dx +
  \int_\Omega\bfS(\nabla\bfu_{m,h}):\nabla\bfxi_h \dx
  &=\int_\Omega\bff_{m}\cdot\bfxi_h \dx.
\end{align}
Notice that we discretize in space and time simultaneously and avoid
an intermediate step with only time or only space discretization. This
has the advantage that we do not need to derive regularity properties
of additional intermediate problems.

Let us compare our discrete equation to the continuous one. We start
with the time steps~$m \geq 2$.  We first take the average over
$(s-\tau,s)$
\begin{align*}
  \frac{u(s)-u(s-\tau)}{\tau} - \dashint_{(s-\tau,s)} \divergence
  \big(\bfS(\nabla \bfu)(\sigma)\big)\,d\sigma
  &=  \dashint_{(s-\tau,s)} \bff(\sigma)\,d\sigma.
\end{align*}
Now, take the mean value over~$J_m$ with respect to~$s$.  Then
\begin{align*}
  d_t \mean{\bfu}_{J_m} - \dashint_{J_m} \dashint_{(s-\tau,s)} \divergence
  \big(\bfS(\nabla \bfu)(\sigma)\big)\,d\sigma\,ds
  &= \dashint_{J_m}
    \dashint_{(s-\tau,s)} \bff(\sigma)\,d\sigma\,ds.
\end{align*}
We define~$\theta_m\,:\, [0,\infty) \to [0,\infty)$ for $m=2,\ldots,M$ by
\begin{align}
  \label{eq:Theta-m}
  \begin{aligned}
  \theta_m(\sigma)
  &:=
    \dashint_\sigma^{\sigma+\tau}
    \frac{\indicator_{J_m}(s)}{2\tau}\,ds
    = \frac 1{2\tau^2}
    \int_{\max \set{\sigma,t_{m-1}}}^{\min \set{\sigma+\tau,t_{m+1}}}
    \,ds
  \\
  &= \frac{(\sigma - t_{m-2}) \indicator_{I_{m-1}} (\sigma)
    + \indicator_{I_m}(\sigma) + (t_{m+1} - \sigma)
    \indicator_{I_{m+1}} (\sigma)}{2 \tau^2}.
  \end{aligned}
\end{align}
Then $\theta_m$ is a weight with total mass one. Let us define
weighted averages by
\begin{align*}
  \mean{\bfg}_{\theta_m} &:= \int_\mathbb{R} \theta_m(s) \bfg(s) \,ds.
\end{align*}
We obtain for $m \geq 2$
\begin{align}
  \label{eq:mean-u-step-later}
  d_t \mean{\bfu}_{J_m} - \divergence
  \big(\mean{\bfS(\nabla \bfu)}_{\theta_m}\big)  &= \mean{\bff}_{\theta_m}.
\end{align}
For $m=1$ we have to proceed slightly differently. We start with
our equation~\eqref{eq:heat2}, take the integral over $(0,s)$ and
divide by~$\tau$ to obtain
\begin{align*}
  \frac{\bfu(s)-\bfu_0}{\tau}- \tau^{-1} \int_{(0,s)} \divergence
  \big(\bfS(\nabla \bfu)(\sigma)\big)\,d\sigma &=  \tau^{-1} \int_{(0,s)}
                                         \bff(\sigma)\,d\sigma. 
\end{align*}
Now, we take the mean value over $J_1$ with respect to~$s$
and obtain
\begin{align*}
  \frac{\mean{\bfu}_{J_1} - \bfu_0}{\tau} - \dashint_{J_1} \tau^{-1} \int_0^s \divergence
  \big(\bfS(\nabla \bfu)(\sigma)\big)\,d\sigma\,ds
  &= \dashint_{J_1} \tau^{-1} \int_0^s \bff(\sigma)\,d\sigma\,ds,
\end{align*}
Let us define the weight $\theta_1 : [0,\infty) \to [0,\infty)$ by
\begin{align} \label{eq:Theta-1} \theta_1(\sigma) := \frac{1}{2\tau^2}
  \int_\sigma^{\infty} \indicator_{J_1}(s) \ds = \frac{2\tau -
    \sigma}{2\tau^2} \indicator_{J_1}(\sigma).
\end{align}
Then $\theta_1$ has total mass one and we may write
\begin{align*}
  \frac{\mean{\bfu}_{J_1} - \bfu_0}{\tau} - \divergence
  \big(\mean{\bfS(\nabla \bfu)}_{\theta_1}\big)
  &=   \mean{\bff}_{\theta_1}.
\end{align*}
With $\mean{\bfu}_{J_0}:= \bfu_0$ the initial equation is now given by
\begin{align} \label{eq:mean-u-step-1}
  d_t \mean{\bfu}_{J_1} - \divergence
  \big(\mean{\bfS(\nabla \bfu)}_{\theta_1}\big)
  &=   \mean{\bff}_{\theta_1}.
\end{align}

\section{Projections operators}
\label{sec:proj-p-lapl}

In this section we consider projections onto the finite element space
$V_h$ introduced in the previous section. In particular, we recall
some known properties of the Scott-Zhang interpolation operator and
prove a gradient estimate for the error of the $L^2$-orthogonal
projection $\Pi_2$ in Theorem~\ref{thm:Vstab-L2}. The latter one is
crucial for the error analysis in the subsequent section.

Let $\PiSZone\,:\, W^{1,1}_0(\Omega) \to V_h$ denote the standard
Scott-Zhang interpolation operator~\cite{ScoZha90} that preserves zero
boundary values, where the values of $\PiSZone \bfv$
on~$\partial \Omega$ are obtained by averaging over edges
in~$\partial \Omega$. Then this operator is stable in~$W^{1,1}$ but
unfortunately not in~$L^1$ or~$L^2$. Therefore, we use slight variant
of the Scott-Zhang operator: Given a function~$\bfv \in L^1(\Omega)$
we extend it by zero outside of~$\Omega$ on an additional layer of
triangles. Now, we take the Scott-Zhang operator that averages only
over $n$-simplices. At the boundary the $n$-simplices, where the
average is calculated, are chosen to lie outside of~$\Omega$. In such
a way we obtain zero boundary values and preserve the~$L^1$-stability,
see the remark after (4.6) in \cite{ScoZha90}. Note that~$\PiSZzero$
does not preserve general polynomial boundary data.

These operators have the following nice properties:
\begin{enumerate}
\item (\textbf{Projection}) $\PiSZzero$ and $\PiSZone$ are linear
  projections onto~$V_h$.
\item \textbf{(Local Stability)} There holds uniformly in $T \in
  \mathcal{T}_h$ 
  \begin{align}
    \label{eq:stab}
    \begin{alignedat}{2}
      \dashint_T \abs{\PiSZzero \bfv}\dx &\lesssim \dashint_{\omega_T}
      \abs{ \bfv}\dx &\qquad &\text{for $\bfv \in L^1(\Omega)$},
      \\
      \dashint_T \abs{\PiSZone \bfv}\dx &\lesssim \dashint_{\omega_T}
      \abs{ \bfv}\dx + \dashint_{\omega_T} h_T \abs{ \nabla \bfv}\dx
      &\qquad &\text{for $\bfv \in W^{1,1}_0(\Omega)$}.
    \end{alignedat}
  \end{align}
\end{enumerate}
It is well-known that these properties imply the following $L^p$
stability results for $1 \leq p < \infty$, e.g. \cite{DR},
\begin{align}
  \label{eq:stabp}
  \begin{aligned}
    \Bigg(\dashint_T \abs{\PiSZzero \bfv}^p\dx\Bigg)^{\frac 1p}
    &\lesssim \Bigg(\, \dashint_{\omega_T}
    \abs{ \bfv}^p \dx \Bigg)^{\frac 1p}
    \\
    \Bigg(\dashint_T \abs{\PiSZone \bfv}^p\dx \Bigg)^{\frac 1p} &\lesssim \Bigg(\,\dashint_{\omega_T}
    \abs{ \bfv}^p \dx \Bigg)^{\frac 1p} + \Bigg(\,\dashint_{\omega_T} h_T^p
    \abs{ \nabla \bfv}^p\dx \Bigg)^{\frac 1p}.
  \end{aligned}
\end{align}
For $p=\infty$ the mean value integrals have to be exchange by maxima.

The following local estimate has been shown by Diening and~\Ruzicka{} in~\cite[Thm. 5.7]{DR}.
\begin{proposition}
  \label{pro:app_V}
  For all $\bfv \in W^{1,p}(\Omega)$ and all
  $T \in \mathcal{T}_h$ it holds that
  \begin{align}
    \label{eq:app_V1}
    \begin{aligned}
      \dashint_T \bigabs{\bfV (\nabla \bfv) - \bfV (\nabla \PiSZone
        \bfv)}^2 \dx &\lesssim
      \inf_{\bfQ\in\mathbb{R}^{N\times
          n}}\dashint_{\omega_T} \bigabs{\bfV(\nabla \bfv) -
        \bfV(\bfQ)}^2 \dx
      \\
      &=
      \dashint_{\omega_T} \bigabs{\bfV(\nabla \bfv) -
        \mean{\bfV(\nabla \bfv)}_{\omega_T}}^2
      \dx.
    \end{aligned}
  \end{align}
  The implicit constant only depends on~$p$ and the
  shape regularity constant~$\gamma$.
\end{proposition}
If follows by a simple application of \Poincare{}'s inequality that
\begin{align}
  \label{eq:PiSZ-local-W12}
  \dashint_T \bigabs{\bfV (\nabla \bfv) - \bfV (\nabla \PiSZone
  \bfv)}^2 \dx &\lesssim h_T^2
                 \dashint_{\omega_T} \bigabs{\nabla{ }\bfV(\nabla \bfv)}^2 \dx
\end{align}
and by summation over all~$T$
\begin{align}
  \label{eq:PiSZ-global-W12}
  \bignorm{\bfV (\nabla \bfv) - \bfV (\nabla \PiSZone
  \bfv)}_{L^2(\Omega)}  &\lesssim h\,
 \bignorm{\nabla \bfV(\nabla \bfv)}_{L^2(\Omega)}.
\end{align}
Let us make a short remark on local estimates in Nikolskii spaces.  For all $g \in N^{\alpha,q}(T)$ with $\alpha \in (0,1]$ and
$q \in [1,\infty)$ it follows by Jensen's inequality and the
definition of $N^{\alpha,q}(T)$ that
\begin{align}
  \label{eq:local-N1}
  \bigg( \dashint_T \abs{g - \mean{g}_T}^q\dx \bigg)^{\frac 1q}
  &\leq \bigg(\frac{1}{\abs{T}^2} \int_{\abs{z} \leq h_T} \int_{T \cap
    (T-z)} \abs{g(y+z)-g(y)}^q \dy \dz \bigg)^{\frac 1q}
  \\
  \label{eq:local-N1b}
  &\lesssim h_T^\alpha
    \frac{\seminorm{g}_{N^{\alpha,q}(T)}}{\abs{T}^{\frac 1 q}}.
\end{align}
It is possible to replace~$T$ by~$\omega_T$.

The next theorem extends~\eqref{eq:PiSZ-local-W12}
and~\eqref{eq:PiSZ-global-W12} to the case of Nikolskii spaces.
\begin{theorem}
  \label{thm:PiSZ-stab}
  Let $\alpha \in (0,1]$. For all ~$T \in \mathcal{T}_h$ it holds that
  \begin{align}
    \label{eq:PiSZ-local-N}
    \left( \dashint_T \bigabs{\bfV (\nabla \bfv) - \bfV (\nabla \PiSZone 
    \bfv)}^2 \dx \right)^\frac{1}{2} &\lesssim h_T^\alpha\,
              \frac{     \left[\bfV(\nabla \bfv)\right]_{N^{\alpha,2}(
                                       \omega_T)}
                                       }{\abs{\omega_T}^{\frac 12}}.
  \end{align}
  Assume additionally that~$\mathcal{T}_h$ is  quasi-uniform. Then we have
  \begin{align}
    \label{eq:PiSZ-global-N}
    \bignorm{\bfV(\nabla \bfv) - \bfV(\nabla \PiSZone \bfv)}_{L^2(\Omega)} \lesssim
    h^{\alpha} [\bfV(\nabla  \bfv)]_{N^{\alpha,2}(\Omega)}.
  \end{align}
\end{theorem}
\begin{proof}
  Estimate~\eqref{eq:PiSZ-local-N} follows directly from
  Proposition~\ref{pro:app_V} and~\eqref{eq:local-N1b}. Similarly
  with~\eqref{eq:local-N1} we obtain
  \begin{align*}
    \lefteqn{\bignorm{\bfV(\nabla \bfv) - \bfV(\nabla \PiSZone \bfv)}_{L^2(\Omega)}^2}
    \qquad &
    \\
           &\lesssim \sum_{T \in \mathcal{T}_h}       \int_{\omega_T} \bigabs{\bfV(\nabla \bfv) -
             \mean{\bfV(\nabla \bfv)}_{\omega_T}}^2
             \dx
    \\
           &\lesssim
             \frac{1}{h^n} \int_{\abs{z} \leq h} \int_{\Omega \cap
             (\Omega-z)} \abs{\bfV(\nabla \bfv)(y+z)-\bfV(\nabla
             \bfv)(y)}^2 \dy  \dz 
    \\
           &\lesssim
             \big(h^{\alpha} [\bfV(\nabla
             \bfv)]_{N^{\alpha,2}(\Omega)} \big)^2.
  \end{align*}
  This proves~\eqref{eq:PiSZ-global-N}.
\end{proof}
\begin{lemma}
  \label{lem:PiSZ-L2-stab}
  Let $\alpha \in (0,1]$. For all ~$T \in \mathcal{T}_h$ it holds that
  \begin{align*}
    \norm{\bfv - \PiSZzero \bfv}_{L^2(\Omega)} \lesssim
    h^{\alpha} [\bfv]_{N^{\alpha,2}(\Omega)}.
  \end{align*}
  
\end{lemma}
\begin{proof}
  Arguing similarly as in the proof of \eqref{eq:PiSZ-global-N} we
  have for all $\bfv\in N^{\alpha,p}(\Omega)$
  \begin{align*}
    \nonumber
    \int_\Omega \left\vert  \bfv-\PiSZzero\bfv\right\vert^2\dx
    &\lesssim\sum_{T\in\mathcal T_h}\int_T \left\vert
      \bfv-\langle\bfv\rangle_{\omega_T}\right\vert^2\dx+\sum_{T\in\mathcal
      T_h}\int_T \left\vert
      \PiSZzero\big(\bfv-\langle\bfv\rangle_{\omega_T}\big)\right\vert^2\dx
    \\
    &\lesssim\sum_{T\in\mathcal T_h}\int_{\omega_T} \left\vert  \bfv-\langle\bfv\rangle_{\omega_T}\right\vert^2\dx\lesssim
      h^{2\alpha}  \left[
      \bfv\right]^2_{N^{\alpha,2}(\Omega)}
  \end{align*}
  using the projection property and the local stability
  estimate~\eqref{eq:stabp}.
\end{proof}
Although the Scott-Zhang operator has wonderful properties it is not
always the best choice for parabolic problems. In particular, the lack
of self-adjointness makes serious problems with the discretization of
term~$\partial_t \bfu$. For the latter one it is much better to use
the~$L^2$-projection $\Pi_2\,:\, L^2(\Omega) \to V_h$.

In fact, we will later use the following
identity for the error~$\bfe_m$ (see Section~\ref{sec:error} for the
exact definition of the error~$\bfe_m$)
\begin{align}
  \label{eq:identity-aux1}
  \begin{aligned}
  \int_\Omega d_t \bfe_m \cdot \Pi_2 &\bfe_m \dx 
  = \int_\Omega d_t \Pi_2 \bfe_m \cdot \Pi_2 \bfe_m \dx\\&
    = \frac{1}{2} d_t \norm{\Pi_2 \bfe_m}_{L^2(\Omega)}^2 + \frac{\tau}{2}
    \norm{d_t \Pi_2 \bfe_m}_{L^2(\Omega)}^2. 
    \end{aligned}
\end{align}
This important identity relies strongly on the self-adjointness
of~$\Pi_2$, which is not available for~$\PiSZone$. This was the reason
for the $h$ and $\tau$ coupling in previous papers.

In the following we will extend~\eqref{eq:PiSZ-global-W12} to the
$L^2$-projection~$\Pi_2$. In particular, we want to prove the
following theorem.
\begin{theorem}
  \label{thm:Vstab-L2}
  Let $\mathcal{T}_h$ be quasi-uniform and $\alpha \in (0,1]$. Then
  \begin{align*}
    \norm{\bfV(\nabla \bfv) - \bfV(\nabla \Pi_2 \bfv))}_{L^2(\Omega)}
    \lesssim h^{\alpha} [\bfV(\nabla \bfv)]_{N^{\alpha,2}(\Omega)}. 
  \end{align*}
\end{theorem}
Before we get to the proof of the theorem let us make a few remarks.
The case~$p=2$ reduces to
\begin{align}
  \label{eq:Pi2forp=2}
  \norm{\nabla \bfv - \nabla \Pi_2 \bfv}_{L^2(\Omega)}
  \lesssim h^{\alpha} [\nabla \bfv]_{N^{\alpha,2}(\Omega)}. 
\end{align}
Since~$\mathcal{T}_h$ is quasi-uniform, this special case can be
easily shown with the help of the Scott-Zhang operator. Indeed, using
$\Pi_2 \PiSZone = \PiSZone$, inverse estimates and the approximation
properties of~$\PiSZone$, in particular \eqref{eq:PiSZ-global-N}, we can estimate
\begin{align*}
  \begin{aligned}
    \norm{\nabla \bfv - \nabla \Pi_2 \bfv}_{L^2(\Omega)} &\leq
    \norm{\nabla (\bfv - \PiSZone \bfv)}_{L^2(\Omega)} + \norm{\nabla
      \Pi_2 (\bfv-\PiSZone \bfv)}_{L^2(\Omega)}
    \\
    &\lesssim \norm{\nabla (\bfv - \PiSZone \bfv)}_{L^2(\Omega)} + h^{-1}
    \norm{ \Pi_2 (\bfv-\PiSZone \bfv)}_{L^2(\Omega)}
    \\
    &\lesssim \norm{\nabla (\bfv - \PiSZone \bfv)}_{L^2(\Omega)} + h^{-1}
    \norm{ \bfv-\PiSZone \bfv}_{L^2(\Omega)}
    \\
    &\lesssim \norm{\nabla (\bfv - \PiSZone \bfv)}_{L^2(\Omega)}
    \\
    & \lesssim h^{\alpha} [\nabla \bfv]_{N^{\alpha,2}(\Omega)}.
  \end{aligned}
\end{align*}
However, in the non-linear case~$p \ne 2$, where $\nabla \bfv$ has to
be replaced by~$\bfV(\nabla \bfv)$, such a simple trick is not
possible. To overcome this problem we will
use sophisticated decay estimates of the~$L^2$-projection which are
due to Eriksson and Johnson~\cite{ErissonJohnson1995} and refined by
Boman~\cite{Boman}. In the following we will derive from their results
decay estimates of the $L^2$-projection for our simple situation of
quasi-uniform meshes.

Let us define the mollifier
\begin{align*}
  \eta(x) &:= c_\nu \exp(-\nu\abs{x}),
\end{align*}
with $c_\nu$ such that $\norm{\eta}_{L^1(\Rn)}=1$. Then
$\eta_h(x) := h^{-n} \eta(x/h)$ satisfies
$\norm{\eta_h}_{L^1(\Rn)}=1$ as well.
\begin{lemma}[Decay estimates of the $L^2$-projection]
  \label{lem:L2-decay}
  Let $\mathcal{T}_h$ be quasi-uniform. Then for every $x \in T$ and
  all $\bfv \in L^1(\Omega)$ and all $\bfw \in W^{1,1}(\Omega)$ we
  have
  \begin{align}
    \label{eq:L2-decay-L}
    \abs{(\Pi_2 \bfv)(x)}
    &\lesssim  
       \big(\eta_h * (\indicator_\Omega \abs{\bfv})\big)(x),
    \\
    \label{eq:L2-decay-W}
    \abs{(\nabla \Pi_2 \bfw)(x)}
    &\lesssim  
       \big(\eta_h * (\indicator_\Omega \abs{\nabla \bfw})\big)(x).
  \end{align}
  The implicit constants only depend on~$n$ and the
  shape regularity~$\gamma$.
\end{lemma}
\begin{proof}
  We begin with the proof of~\eqref{eq:L2-decay-L}
  \begin{align*}
    \abs{(\Pi_2 \bfv)(x)}
    &\lesssim \dashint_T \abs{\Pi_2 \bfv}\dy
    \\
    &\lesssim \int_\Omega \eta_h(x-y)  \abs{(\Pi_2 \bfv)(y)}\dy =
      \norm{\eta_h(x-\cdot) \Pi_2 \bfv(\cdot)}_{L^1(\Omega)}.
  \end{align*}
  Since our triangulation is quasi-uniform we may choose a constant
  regularized mesh function~$h(x) := h$ in order to apply the results
  of~\cite{ErissonJohnson1995} and~\cite{Boman}.  In particular, by
  Lemma~2.3 of~\cite{Boman} (applied to the case~$p=1$)
  it follows that
  \begin{align*}
    \abs{(\Pi_2 \bfv)(x)}
    &\lesssim       \norm{\eta_h(x-\cdot) 
      \bfv(\cdot)}_{L^1(\Omega)}
      \lesssim \big(\eta_h * (\indicator_\Omega \abs{\bfv})\big)(x).
  \end{align*}
  This proves~\eqref{eq:L2-decay-L}.  Let us remark that the results
  of Boman are unfortunately not properly displayed. In particular,
  they define
  \begin{align*}
    \delta_{\mathcal{T}_h} &:=  \max_{T \in \mathcal{T}_h} \max_{T'
                             \in \omega_T} \abs{1 - h_T^2/h_{T'}^2}.
  \end{align*}
  Thus, only a uniform mesh gives~$\delta_{\mathcal{T}_h}= 0$. It
  would have been better to use
  \begin{align*}
    \delta_{\mathcal{T}_h} &:=  \max_{T \in \mathcal{T}_h} \max_{x \in
                             T} \max_{y \in \omega_T} \abs{1 - h(x)^2/h(y)^2},
  \end{align*}
  which is zero for all quasi-uniform meshes with
  constant regularized mesh function. In the paper of Eriksson and
  Johnson~\cite{ErissonJohnson1995}  this was done properly and the
  case of quasi-uniform meshes is included. A careful inspection of
  the proofs by Boman shows that it is enough to use the alternative
  definition of~$\delta_{\mathcal{T}_h}$, so that quasi-uniform meshes
  are included.

  The proof of~\eqref{eq:L2-decay-W} is analogously using   Lemma~2.5
  of Boman~\cite{Boman}, i.e.
  \begin{align*}
    \norm{\eta_h(x-\cdot) \nabla \Pi_2 \bfw(\cdot)}_{L^1(\Omega)}
    &\lesssim
      \norm{\eta_h(x-\cdot) \nabla
      \bfw(\cdot)}_{L^1(\Omega)}.
  \end{align*}
  This proves the claim.
\end{proof}
Closely related to \eqref{eq:def-S} and \eqref{eq:def-V} is the
shifted Orlicz function $\varphi_a$ defined by
\begin{align} \label{eq:def-phi}
  \varphi_a(t):= \int_0^t (\kappa + a + s)^{p-2}s \ds
\end{align}
for $a \geq 0$ (see also the appendix).  We are now prepared for the
proof of Theorem~\ref{thm:Vstab-L2}.
\begin{proof}[Proof of Theorem~\ref{thm:Vstab-L2}]
  We estimate using Lemma~\ref{lem:hammer}
  \begin{align*}
    \I &:= \int_\Omega \abs{\bfV(\nabla \bfv) - \bfV(\nabla \Pi_2 \bfv)}^2
         \dx
    \\
       &\hphantom{:} \eqsim \int_\Omega \varphi_{\abs{\nabla \bfv}}\left(\abs{ \nabla
         \bfv - \nabla \Pi_2 \bfv} \right) \dx
    \\
       &\hphantom{:}\lesssim \int_\Omega \varphi_{\abs{\nabla \bfv}}\left(\abs{ \nabla
         \bfv - \nabla \PiSZone \bfv} \right) \dx  +  \int_\Omega
         \varphi_{\abs{\nabla \bfv}}\left(\abs{ 
         \nabla \Pi_2( \bfv -\PiSZone \bfv)} \right) \dx
    \\
       &\hphantom{:}=: \textrm{II} + \textrm{III}.
  \end{align*}
  Now, by Lemma~\ref{lem:hammer} and Theorem~\ref{thm:PiSZ-stab}
  \begin{align*}
    \textrm{II} &\eqsim  \int_\Omega \abs{\bfV(\nabla \bfv) -
                  \bfV(\nabla \PiSZone \bfv)}^2 \dy \lesssim                   h^{2\alpha} \left[\bfV( \nabla \bfv) \right]_{N^{\alpha,2}(\Omega)}^2.
  \end{align*}
  Moreover, by Lemma~\ref{lem:L2-decay}, Jensen's inequality
  \begin{align*}
    \textrm{III}
    &
      \lesssim \int_\Omega \phi_{\abs{\nabla \bfv(x)}} \big(\eta_h *
      (\indicator_\Omega \abs{\nabla (\bfv - \PiSZone \bfv})\big)(x) \dx
    \\
    &= \int_\Omega \phi_{\abs{\nabla \bfv(x)}} \bigg( \int_\Omega \eta_h(x-y)
      \abs{\nabla (\bfv - \PiSZone \bfv)(y)} \dy\bigg) \dx
    \\
    &\lesssim \int_\Omega  \int_\Omega \eta_h(x-y) \phi_{\abs{\nabla \bfv(x)}} \big(
      \abs{\nabla (\bfv - \PiSZone \bfv)(y)}\big) \dy \dx.
 \intertext{Now, by the shift-change Lemma~\ref{lem:shift_ch} and Lemma~\ref{lem:hammer}}
    &\lesssim \int_\Omega  \int_\Omega \eta_h(x-y) \phi_{\abs{\nabla \bfv(y)}} \big(
      \abs{\nabla (\bfv - \PiSZone \bfv)(y)}\big) \dy \dx
    \\
    &\qquad +
      \int_\Omega  \int_\Omega \eta_h(x-y) \abs{\bfV(\nabla \bfv)(x) -
      \bfV(\nabla \bfv)(y)}^2\, \dy \dx
    \\
    &\lesssim \int_\Omega \abs{\bfV(\nabla \bfv) - \bfV(\nabla \PiSZone \bfv)}^2 \dy
    \\
    &\qquad +
      \int_\Omega  \int_\Omega \eta_h(x-y) \abs{\bfV(\nabla \bfv)(x) -
      \bfV(\nabla \bfv)(y)}^2\, \dy \dx
    \\
    &=: \textrm{III}_1 + \textrm{III}_2.
  \end{align*}
  Again by Theorem~\ref{thm:PiSZ-stab}
  \begin{align*}
    \textrm{III}_1 &\lesssim 
                     h^{2\alpha} \left[\bfV( \nabla \bfv) \right]_{N^{\alpha,2}(\Omega)}^2.
  \end{align*}
  We estimate further
  \begin{align*}
    \textrm{III}_2 &=       \int_\Omega  \int_\Omega \eta_h(x-y) \abs{\bfV(\nabla \bfv)(x) -
                     \bfV(\nabla \bfv)(y)}^2\, \dy \dx
    \\
                   &=  \int_{\Rn}  \int_{\Omega \cap
                     (\Omega-z)} \eta_h(z) \abs{\bfV(\nabla \bfv(y+z))
                     -\bfV(\nabla \bfv(y)) }^2 \dy\dz
    \\
                   &\lesssim \int_{\Rn}  \eta_h(z)
                     \abs{z}^{2\alpha}
                     \dz \left[\bfV(
                     \nabla \bfv) \right]_{N^{\alpha,2}(\Omega)}^2
    \\
                   &= h^{2\alpha} \int_{\R^n}  \eta(z)
                     \abs{z}^{2\alpha} \dz
                     \left[\bfV( \nabla \bfv)
                     \right]_{N^{\alpha,2}(\Omega)}^2
    \\
                   &\lesssim h^{2\alpha} \left[\bfV( \nabla \bfv) \right]_{N^{\alpha,2}(\Omega)}^2,
  \end{align*}
  which proves the claim.
\end{proof}

\section{Error analysis}
\label{sec:error}

Our aim now is to establish the convergence rate of the difference
between the solution to the continuous problem
solving~\eqref{eq:heatweak1}, and that of the discrete
problem~\eqref{eq:tdiscr}. 
To do this, we first collect the following assumptions on the data. Throughout the rest of this section we assume that
\begin{align}\label{assumpt} 
  \bfu_0\in L^2(\Omega), \quad \bff\in L^{p'}(0,T;W^{-1,p'}(\Omega)).
\end{align}
Recall the definition of weighted averages
\begin{align*}
\mean{\bfg}_{\theta_m} := \int_\mathbb{R} \theta_m(\sigma) \bfg(\sigma) \dsigma
\end{align*} 
with $\theta_m$ given by \eqref{eq:Theta-m} resp. \eqref{eq:Theta-1}.

In the following we state the main result of the paper.
\begin{theorem}\label{thm:main}
  Suppose that \eqref{assumpt} holds. Let $\bfu$ be the unique weak
  solution to \eqref{eq:heat} in the sense of Definition
  \ref{def:weak}.  Assume that
  \begin{align}
    \label{ass:reg}
    \begin{aligned}
      \bfV(\nabla\bfu)&\in L^2(I;\,N^{\alpha_x,2}(\Omega)) \,\cap \,
      N^{\alpha_t,2}(I;L^2(\Omega)) ,
      \\
      \bfu&\in L^\infty(I;N^{\alpha_x,2}(\Omega))
    \end{aligned}
  \end{align}
  for some $\alpha_x,\alpha_t\in(0,1]$. Let $\mathcal{T}_h$ be
  quasi-uniform.  Then we have uniformly in $\tau$ and $h$
  \begin{align*}
    \max\limits_{1\leq m \leq M}
    &\norm{\mean{\bfu}_{J_m} -
      \bfu_{m,h}}_{L^2(\Omega)}^2 +
      \sum_{m=1}^M \, \int_{J_m} \norm{
      \bfV(\nabla \bfu(s)) - \bfV(\nabla
      \bfu_{m,h})}_{L^2(\Omega)}^2 \dd s
    \\
    &\lesssim h^{2\alpha_x} \left( \sup_{s \in [0,T]}
      \seminorm{\bfu(s)}_{N^{\alpha_x,2}(\Omega)}^2 +  \int_0^T \left[
      \bfV(\nabla \bfu(s) ) \right]_{N^{\alpha_x,2}(\Omega)}^2 \dd
      s\right)
    \\
    &+ \tau^{2 \alpha_t} 
      \left[ \bfV(\nabla \bfu)
      \right]_{N^{\alpha_t,2}(I;L^2(\Omega))}^2
,
  \end{align*}
  where $\bfu_{m,h}$ is the solution to \eqref{eq:tdiscr-pre} with
  $\bff_m = \mean{\bff}_{\theta_m}$ and $\bfu_{0,h} := \Pi_2 \bfu_0$,
  where $\Pi_2$ is the~$L^2$-projection to~$V_h$. The  hidden constant
  is independent of~$T$.
\end{theorem}

\begin{proof}
  Define the averaged error by $\bfe_m := \mean{\bfu}_{J_m} - \bfu_{m,h}$.

  Recall the solution $\bfu$ satisfies \eqref{eq:mean-u-step-later} and \eqref{eq:mean-u-step-1}, whereas the discrete solution $\bfu_{m,h}$ satisfies \eqref{eq:tdiscr-pre}. Therefore it holds for $m=1,\ldots,M$ as an equation in $V_h^*$
  \begin{align} \label{eq:Error}
    \begin{aligned}
      d_t \bfe_m &- \Div \left( \mean{\bfS(\nabla \bfu) }_{\theta_m} -
        \bfS(\nabla \bfu_{m,h}) \right)
      = \mean{\bff}_{\theta_m} - \bff_m
      = 0.
    \end{aligned}
  \end{align}

  Let us also recall that the $L^2$-projection $\Pi_2 : L^2(\Omega) \to V_h$ is defined by 
  \begin{align} \label{eq:L2proj}
    (\Pi_2 \bfv - \bfv, \bfxi_h ) = 0  
  \end{align}
  for all $\bfxi_h \in V_h$. In particular for $\bfxi_h = \Pi_2 \bfv$ it holds 
  \begin{align} \label{eq:L2projb}
    (\Pi_2 \bfv , \Pi_2 \bfv ) = ( \bfv, \Pi_2 \bfv ).
  \end{align}

  Let $m \in \set{1,\ldots,M}$ and choose $\bfxi_h = \Pi_2 \bfe_m$ in \eqref{eq:Error} and get
  \begin{align*}
    \I + \II &:= \int_\Omega d_t \bfe_m \cdot \Pi_2 \bfe_m \dx + \int_\R \int_\Omega \theta_m(\sigma) (\bfS(\nabla \bfu ( \sigma)) - \bfS(\nabla \bfu_{m,h}) ): \nabla \bfe_m \dx \dsigma \\
             &= \int_\R \int_\Omega \theta_m(\sigma) (\bfS(\nabla \bfu ( \sigma)) - \bfS(\nabla \bfu_{m,h}) ): \nabla (\bfe_m - \Pi_2 \bfe_m) \dx \dsigma \\
             &=: \III.
  \end{align*}
  Keep in mind that the weights $\theta_1$ and $\theta_m$ are
  supported in $I_1 \cup I_2$ and $I_{m-1}\cup I_m \cup I_{m+1}$
  respectively, see \eqref{eq:Theta-m} and
  \eqref{eq:Theta-1}. Consequently, the integrals over $\R$ are
  well-defined.
  
  Due to \eqref{eq:L2projb} and \eqref{eq:dtaa} the first term can be written as
  \begin{align*}
    \I &= \int_\Omega d_t \Pi_2 \bfe_m \cdot \Pi_2 \bfe_m \dx = \frac{1}{2} d_t \norm{\Pi_2 \bfe_m}_{L^2(\Omega)}^2 + \frac{\tau}{2} \norm{d_t \Pi_2 \bfe_m}_{L^2(\Omega)}^2.
  \end{align*}
  We split the second term into two parts and use Lemma \ref{lem:hammer}
  \begin{align*}
    \II &= \int_\R \dashint_{J_m} \int_\Omega \theta_m(\sigma) (\bfS(\nabla \bfu ( \sigma)) - \bfS(\nabla \bfu_{m,h}) ): \nabla ( \bfu(s) - \bfu_{m,h} ) \dx \dd s  \dsigma  \\
        &\eqsim \dashint_{J_m} \int_\Omega  \abs{\bfV(\nabla \bfu(s)) - \bfV(\nabla \bfu_{m,h}) }^2 \dx \dd s  \\
        &\quad + \int_\R \dashint_{J_m} \int_\Omega\theta_m(\sigma) (\bfS(\nabla \bfu ( \sigma)) - \bfS(\nabla \bfu(s)) ): \nabla ( \bfu(s) - \bfu_{m,h} ) \dx  \dd s \dsigma \\
        &=: \II[1] + \II[2].
  \end{align*}
  The latter is estimated using Lemma \ref{lem:young2}
  \begin{align*}
    \II[2] &\leq \delta \dashint_{J_m} \int_\Omega  \abs{\bfV(\nabla \bfu(s)) - \bfV(\nabla \bfu_{m,h}) }^2 \dx \dsigma \\
           &\quad + c_\delta \int_\R \dashint_{J_m} \int_\Omega \theta_m(\sigma) \abs{ \bfV( \nabla \bfu(s) ) - \bfV( \nabla \bfu(\sigma) ) }^2 \dx \dd s \dsigma.
  \end{align*}
  The next step is to decompose $\III$ in an analogous way. Due to Lemma \ref{lem:hammer} and \ref{lem:young2} it holds
  \begin{align*}
    \III &=  \dashint_{J_m} \int_\Omega (\bfS(\nabla \bfu ( s)) - \bfS(\nabla \bfu_{m,h}) ): \nabla (\bfu(s) - \Pi_2 \bfu(s) ) \dx \dd s  \\
         &\quad + \int_\R \dashint_{J_m} \int_\Omega \theta_m(\sigma) (\bfS(\nabla \bfu ( \sigma)) - \bfS(\nabla \bfu(s)) ): \nabla (\bfu(s) - \Pi_2 \bfu(s) ) \dx \dd s \dsigma \\
         &\leq \delta \dashint_{J_m} \int_\Omega \abs{\bfV(\nabla \bfu(s)) - \bfV(\nabla \bfu_{m,h}) }^2 \dx \dd s \\
         &\quad + c_\delta  \dashint_{J_m} \int_\Omega  \abs{\bfV(\nabla \bfu(s)) - \bfV(\nabla \Pi_2 \bfu(s))}^2 \dx \dd s \\
         &\quad + c_\delta \int_\R \dashint_{J_m} \int_\Omega
           \theta_m(\sigma) \abs{\bfV(\nabla \bfu(\sigma)) -
           \bfV(\nabla \bfu(s))}^2 \dx \dd s \dsigma
    \\
    &=: \III[1] + \III[2] + \III[3].
  \end{align*}
  Recall the definition of the weights \eqref{eq:Theta-m} and \eqref{eq:Theta-1} and estimate
  \begin{align} \label{eq:VNikol}
    \begin{aligned}
      \lefteqn{\III[2] \leq \int_\R \dashint_{J_m} \int_\Omega \theta_m(\sigma)
        \abs{ \bfV( \nabla \bfu(s) ) - \bfV( \nabla \bfu(\sigma) ) }^2
        \dx \dd s \dsigma} \qquad &
      \\
      &\lesssim \frac{1}{\tau} \int_{ \abs{z} \leq \tau } \int_{ \text{supp}\,\theta_m } \norm{ \bfV( \nabla \bfu(\sigma + z ) ) - \bfV( \nabla \bfu(\sigma) ) }_{L^2(\Omega)}^2  \dsigma \dz.
    \end{aligned}
  \end{align}
  Now, sum over $m$ and multiply by $\tau$,
  \begin{align*}
    &\norm{\Pi_2 \bfe_m}_{L^2(\Omega)}^2 + \tau \sum_{l=1}^m \tau \norm{d_t \Pi_2\bfe_l}_{L^2(\Omega)}^2 \\
    &\quad +  \tau \sum_{l=1}^m \dashint_{J_l} \int_\Omega \abs{\bfV(\nabla \bfu(s)) - \bfV(\nabla \bfu_{l,h}) }^2 \dx \dd s \\
    &\lesssim  \tau \sum_{l=1}^m \dashint_{J_l} \int_\Omega \abs{\bfV(\nabla \bfu(s)) - \bfV(\nabla \Pi_2 \bfu(s))}^2 \dx \dd s \\
    &\quad + \tau \sum_{l=1}^m \int_\R \dashint_{J_l} \int_\Omega \theta_l(\sigma) \abs{\bfV(\nabla \bfu(\sigma)) - \bfV(\nabla \bfu(s))}^2 \dx \dd s \dsigma
    \\
    &=: K_1 + K_2.
  \end{align*}
  The non-linear stability Theorem \ref{thm:Vstab-L2} for the $L^2$-projection yields
  \begin{align*}
    K_1 &\lesssim h^{2\alpha_x} \int_0^T \left[ \bfV(\nabla \bfu(s) ) \right]_{N^{\alpha_x,2}(\Omega)}^2 \dd s.
  \end{align*}
  The second term is bounded using \eqref{eq:VNikol} by
  \begin{align*}
    K_2 &\lesssim  \sum_{l=1}^m  \int_{ \abs{z} \leq \tau } \int_{\text{supp}\, \theta_l } \norm{ \bfV( \nabla \bfu(\sigma + z ) ) - \bfV( \nabla \bfu(\sigma) ) }_{L^2(\Omega)}^2  \dsigma \dz \\
        &\lesssim \tau^{2 \alpha_t} \left[ \bfV(\nabla \bfu) \right]_{N^{\alpha_t,2}(I;L^2(\Omega))}^2.
  \end{align*}
  We have obtained uniform estimates for the projected
  averaged error $\Pi_2 \bfe_m$
  \begin{align*}
    \lefteqn{\norm{\Pi_2 \bfe_m}_{L^2(\Omega)}^2 + \tau \sum_{l=1}^m \dashint_{J_l} \int_\Omega \abs{\bfV(\nabla \bfu(s)) - \bfV(\nabla \bfu_{l,h}) }^2 \dx \dd s} \quad
    \\
    &\lesssim
      h^{2\alpha_x} \int_0^T \left[ \bfV(\nabla \bfu(s) )
      \right]_{N^{\alpha_x,2}(\Omega)}^2 \dd s +
      \tau^{2 \alpha_t} \left[ \bfV(\nabla \bfu) \right]_{N^{\alpha_t,2}(I;L^2(\Omega))}^2.
  \end{align*}
  However, we are interested in the error
  $\bfe_m:= \mean{\bfu}_{J_m} - \bfu_{m,h}$. Thus, we estimate
  \begin{align*}
    \norm{\bfe_m}_{L^2(\Omega)}^2 &\lesssim \norm{\mean{\bfu}_{J_m} -
                                    \Pi_2 \mean{\bfu}_{J_m}}_{L^2(\Omega)}^2 + \norm{\Pi_2 \bfe_m}_{L^2(\Omega)}^2 \\
                                  &=: K_3 + K_4.
  \end{align*}
  The first term can be bounded using Jensen's inequality and Lemma~\ref{lem:PiSZ-L2-stab}
  \begin{align*}
    K_3 &\lesssim \dashint_{J_m} \norm{\bfu(s) - \Pi_2
          \bfu(s)}_{L^2(\Omega)}^2 \dd s
          \leq \dashint_{J_m} \norm{\bfu(s) - \PiSZzero
          \bfu(s)}_{L^2(\Omega)}^2 \dd s
    \\
        &\lesssim h^{2\alpha_x} \sup_{s \in J_m} \seminorm{\bfu(s)}_{N^{\alpha_x,2}(\Omega)}^2.
  \end{align*}
  The bound for $K_4$ has already been established.

  Collecting all terms and taking the maximum over $m$, we arrive at the desired estimate
  \begin{align*}
    \max\limits_{1\leq m \leq M} &\norm{\mean{\bfu}_{J_m} - \bfu_{m,h}}_{L^2(\Omega)}^2 +  \sum_{m=1}^M \, \int_{J_m} \norm{ \bfV(\nabla \bfu(s)) - \bfV(\nabla \bfu_{m,h})}_{L^2(\Omega)}^2 \dd s  \\
                                 &\lesssim h^{2\alpha_x} \left( \sup_{s \in [0,T]} \seminorm{\bfu(s)}_{N^{\alpha_x,2}(\Omega)}^2 +  \int_0^T \left[ \bfV(\nabla \bfu(s) ) \right]_{N^{\alpha_x,2}(\Omega)}^2 \dd s\right) \\
                                 &\quad + \tau^{2 \alpha_t}
                                   \left[ \bfV(\nabla \bfu)
                                   \right]_{N^{\alpha_t,2}(I;L^2(\Omega))}^2.
  \end{align*}
  This proves the claim.
\end{proof}
It is not that important that we choose
$\bff_m = \mean{\bff}_{\theta_m}$. Indeed, we can allow for a certain
class of discrete forces $\bff_m$ in our numerical scheme and still
have convergence of order $\alpha_t$.
\begin{corollary} \label{cor:Main-2} Let the assumption of Theorem
  \ref{thm:main} be satisfied. Additionally, assume that there exists
  $c_\bff \geq 0$ independent of $\tau$ such that
  \begin{align} \label{eq:reg-f}
    \tau \sum_{m=1}^M \norm{\mean{\bff}_{\theta_m} - \bff_m }_{L^2(\Omega)}^2 \leq c_\bff \tau^{2\alpha_t}.
  \end{align}
  Then we have uniformly in $\tau$ and $h$
  \begin{align*}
    \lefteqn{\max\limits_{1\leq m \leq M}
    \norm{\mean{\bfu}_{J_m} - \bfu_{m,h}}_{L^2(\Omega)}^2 +
    \sum_{m=1}^M \, \int_{J_m} \norm{ \bfV(\nabla \bfu(s)) -
    \bfV(\nabla \bfu_{m,h})}_{L^2(\Omega)}^2 \dd s}
    \quad&
    \\
         &\lesssim h^{2\alpha_x} \left( \sup_{s \in [0,T]}
           \seminorm{\bfu(s)}_{N^{\alpha_x,2}(\Omega)}^2 +  \int_0^T \left[
           \bfV(\nabla \bfu(s) ) \right]_{N^{\alpha_x,2}(\Omega)}^2 \dd
           s\right)
    \\
         &\quad + \tau^{2 \alpha_t} \!\left(\!
           \left[ \bfV(\nabla \bfu)
           \right]_{N^{\alpha_t,2}(I;L^2(\Omega))}^2
           +  c_{\bff}\right).
  \end{align*}
  In particular, we can use $\bff_m:= \bff(t_m)$ for
  $\bff \in \mathcal C^{0,\alpha_t}(0,T;L^2(\Omega))$.  The  hidden constant
  may depend exponentially on~$T$.
\end{corollary}
\begin{proof}
  The proof is essentially the same as for Theorem \ref{thm:main}. But we do not have the the cancellation of $\mean{\bff}_{\theta_m}$ and $\bff_m$ as in \eqref{eq:Error}. Instead it holds for all $m \in \set{1,\ldots,M}$ and $\bfxi_h \in V_h$
  \begin{align} \label{eq:GalOrt-f}
    \begin{aligned}
      &\int_\Omega d_t \bfe_m \cdot \bfxi_h \dx + \int_\R \int_\Omega \theta_m(\sigma)(\bfS(\nabla \bfu(\sigma)) - \bfS(\nabla \bfu_{m,h}) ) \cdot \nabla \bfxi_h \dx \dsigma \\
      &\quad = \int_\R \int_\Omega  \theta_m(\sigma) \left( \bff(\sigma) - \bff_m \right)  \cdot \bfxi_h \dx\dsigma.
    \end{aligned}
  \end{align}
  We choose $\xi_h = \Pi_2 \bfe_m$ and proceed as in the proof of
  Theorem \ref{thm:main}. The additional term involving the difference
  in $\bff$ is bounded using Hölder's and Young's inequality by
  \begin{align*}
    \int_\Omega \int_\R \theta_m(\sigma) \left( \bff(\sigma) - \bff_m
    \right) \dsigma \cdot \Pi_2 \bfe_m \dx
    &\leq
      \norm{\mean{\bff}_{\theta_m}
      - \bff_m}_{L^2(\Omega)}
      \norm{\Pi_2
      \bfe_m}_{L^2(\Omega)}
    \\
    &\leq  \norm{\mean{\bff}_{\theta_m} - \bff_m}_{L^2(\Omega)}^2 +
      \norm{\Pi_2 \bfe_m}_{L^2(\Omega)}^2 . 
  \end{align*}
  The summation over $m$ and the multiplication by $\tau$ yield
  \begin{align*}
    &\norm{\Pi_2 \bfe_m}_{L^2(\Omega)}^2 + \tau \sum_{l=1}^m \tau
      \norm{d_t \Pi_2\bfe_l}_{L^2(\Omega)}^2
    \\
    &\quad +  \tau \sum_{l=1}^m \dashint_{J_l} \int_\Omega \abs{\bfV(\nabla \bfu(s)) - \bfV(\nabla \bfu_{l,h}) }^2 \dx \dd s \\
    &\lesssim  \tau \sum_{l=1}^m \dashint_{J_l} \int_\Omega \abs{\bfV(\nabla \bfu(s)) - \bfV(\nabla \Pi_2 \bfu(s))}^2 \dx \dd s \\
    &\quad + \tau \sum_{l=1}^m \int_\R \dashint_{J_l} \int_\Omega
      \theta_l(\sigma) \abs{\bfV(\nabla \bfu(\sigma)) - \bfV(\nabla
      \bfu(s))}^2 \dx \dd s \dsigma
    \\
    &\quad  + \tau \sum_{l=1}^m \norm{\mean{\bff}_{\theta_l} - \bff_l}_2^2 +  \tau \sum_{l=1}^m  \norm{\Pi_2 \bfe_l}_{L^2(\Omega)}^2.
  \end{align*}
  The same arguments as in Theorem~\ref{thm:main}, an application of
  the discrete Gronwall inequality (with constants that depend
  exponentially on the time horizon $T$) and~\eqref{eq:reg-f} allow
  to close the argument.
\end{proof}
  \begin{remark}
    \label{rem:u-hoelder}
    If additionally
    $\bfu \in\mathcal{C}^{0,\alpha_t}(\overline I;L^2(\Omega))$, then
    we can replace
    \begin{align*}
      \max_{1\leq m \leq M} \norm{\mean{\bfu}_{J_m} -
      \bfu_{m,h}}_{L^2(\Omega)}^2
    \end{align*}
    in the error estimate of Theorem~\ref{thm:main} and
    Corollary~\ref{cor:Main-2} by the point-wise error
    \begin{align*}
      \max_{1\leq m \leq M} \norm{\bfu(t_m) -
      \bfu_{m,h}}_{L^2(\Omega)}^2.
    \end{align*}
    This follows immediately from
    \begin{align*}
      \norm{\bfu(t_m) -
      \bfu_{m,h}}_{L^2(\Omega)} 
      &\lesssim  \norm{\mean{\bfu}_{J_m} -
        \bfu_{m,h}}_{L^2(\Omega)} + \tau^{\alpha_t} \left[\bfu
        \right]_{\mathcal{C}^{0,\alpha_t}(\overline 
        I;L^2(\Omega))}.
  \end{align*}
\end{remark}

\section{Numerical Experiments}
\label{sec:exp}
This section underlines the theoretical results of this paper and exploits the practical behaviour of the proposed numerical scheme. The experiments base on the open source tool for solving partial differential equations FEniCS \cite{LoggMardalWells12}. The supplementary material of this paper contains the implementations of the following three experiments.

\subsection{Constant force on slit domain}\label{subsec:Slit}
The first experiment approximates the solution to the $p$-heat equation \eqref{eq:heat} on the slit domain $\Omega = (-1,1)^2\setminus (-1,0] \times \lbrace 0 \rbrace$ with constant right-hand side $\bff\equiv 2$. The experiment applies the numerical scheme of this paper with the discrete spaces $V_h$ from \eqref{def:Vh} with $r=1$ (piece-wise affine polynomials) and $r=2$ (piece-wise quadratic polynomials) for a sequence of uniformly refined meshes $\mathcal{T}_h$ and halved time steps $\tau$. For $r=1$, the smallest space $V_h$ is of dimension $\dim V_h =\textup{ndof} = 12$ and the computation utilizes $M = 4$ time steps, the largest space $V_h$ is of dimension $\dim V_h = \textup{ndof} = 82689$ and the computation utilizes $M = 512$ time steps ($33 \leq \dim V_h \leq 82689$ and $4 \leq M \leq 256$ for $r=2$). 
Since the exact solution to this problem is unknown, we compare the solutions $\bfu_{m,h} \in V_h$ to a reference solution $\bfu_\textup{ref} \in L^2(I;V_h^\textup{ref})$.
The space $V_h^\textup{ref}$ is of polynomial degree $r+1$ and dimension $\dim V_h^\textup{ref} = 1337107$ for $r=1$ and $\dim V_h = 875467$ for $r=2$. The underlying triangulation of $V_h^\textup{ref}$ results from an adaptive finite element loop (see \cite[Sec.\ 5.1]{CarstensenFeischlPagePraetorius14}) for the Poisson model problem $-\Delta \bfu \equiv 1$ in  $H^{-1}(\Omega)$. The reference solution $\bfu_\textup{ref}$ results from this paper's scheme with $M = 1024$ ($r=1$) and $M=512$ ($r=2$) time steps.
We expect from regularity
theory that
\begin{align}\label{eq:regExp1}
\begin{aligned}
  \bfV(\nabla\bfu) &\in
                     L^2(I;N^{\frac 12,2}(\Omega))\cap
                     N^{1,2}(I;L^2(\Omega)),
  \\
    \bfu &\in
    L^\infty(I;N^{1,2}(\Omega))\cap
    \mathcal{C}^{0,1}(\overline{I};L^2(\Omega)).
\end{aligned}
\end{align}
Thus, the solution is smooth in time ($\alpha_t = 1$) but rough in space ($\alpha_x = 1/2$). The convergence history plots in Figure \ref{fig:slit1} displays the convergence of the error contributions 
\begin{align}\label{eq:DefErrors}
\begin{aligned}
\textup{err}_{L^\infty L^2} & :=\max_{1\leq m \leq M} \lVert \bfu_{h,m} - \langle \bfu_\textup{ref}\rangle_{J_m}\rVert_{L^2(\Omega)}^2,\\
\textup{err}_{L^2 \bfV} &: = \tau \sum_{m=1}^M \lVert \bfV (\nabla \bfu_{h,m}) -\bfV(\nabla \langle \bfu_\textup{ref}\rangle_{J_m})\rVert_{L^2(\Omega)}^2 \\
& \hphantom{:}\approx  \tfrac 12 \sum_{m=1}^M \int_{J_m} \lVert \bfV (\nabla \bfu_{h,m}) -\bfV(\nabla \bfu_\textup{ref})\rVert_{L^2(\Omega)}^2\,\mathrm{d}\sigma.  
\end{aligned}
\end{align}
The resulting rate is in agreement with the a priori estimate in Theorem \ref{thm:main}, that is, the error $\textup{err}_{L^2\bfV} \eqsim h \eqsim \tau$. The fast convergence of the error $\textup{err}_{L^\infty L^2} \eqsim h^2 \eqsim \tau^2$ might result from the higher regularity of $\bfu$ in \eqref{eq:regExp1}.
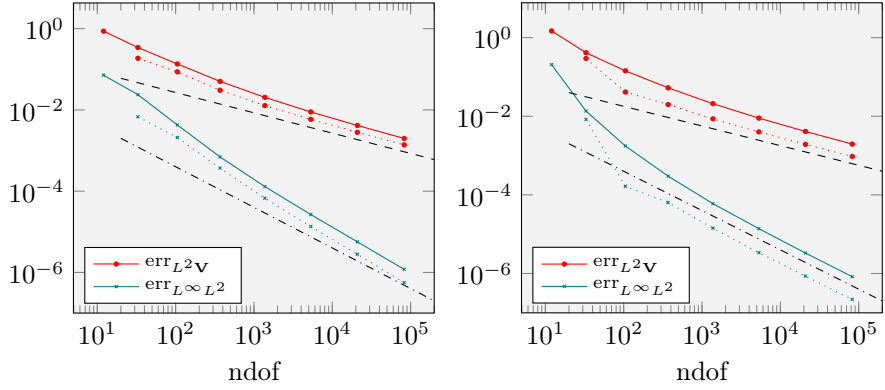
\begin{figure}
\begin{center}
\begin{tikzpicture}
\begin{axis}[
clip=false,
width=.5\textwidth,
height=.45\textwidth,
xmode = log,
ymode = log,
xlabel={ndof},
cycle multi list={\nextlist MyColors},
scale = {1},
ymin = {1e-7},
clip = true,
legend cell align=left,
legend style={legend columns=1,legend pos= south west,font=\fontsize{7}{5}\selectfont}
]
	\addplot table [x=ndof,y=sqIntegralError] {Data/Exp1/Experiment1_deg_1_degRef_2_p_1.5.txt};
		\addplot table [x=ndof,y=sqLinftyError] {Data/Exp1/Experiment1_deg_1_degRef_2_p_1.5.txt};
		\pgfplotsset{cycle list shift=1}
	\addplot table [x=ndof,y=sqIntegralError] {Data/Exp1/Experiment1_deg_2_degRef_3_p_1.5.txt};	
		\addplot table [x=ndof,y=sqLinftyError] {Data/Exp1/Experiment1_deg_2_degRef_3_p_1.5.txt};
		\addplot[dashed,sharp plot,update limits=false] coordinates {(2e1,6e-2) (2e5,6e-4)};\label{plot:hinv}
		\addplot[dash dot,sharp plot,update limits=false] coordinates {(2e1,2e-3) (2e5,2e-7)};\label{plot:hhalfinv}
	\legend{
	{$\textup{err}_{L^2 \bfV}$},
	{$\textup{err}_{L^\infty L^2}$},};
\end{axis}
\end{tikzpicture}
\begin{tikzpicture}
\begin{axis}[
clip=false,
width=.5\textwidth,
height=.45\textwidth,
xmode = log,
ymode = log,
xlabel={ndof},
cycle multi list={\nextlist MyColors},
scale = {1},
ymin = {1e-7},
clip = true,
legend cell align=left,
legend style={legend columns=1,legend pos= south west,font=\fontsize{7}{5}\selectfont}
]
	\addplot table [x=ndof,y=sqIntegralError] {Data/Exp1/Experiment1_deg_1_degRef_2_p_3.0.txt};
	\addplot table [x=ndof,y=sqLinftyError] {Data/Exp1/Experiment1_deg_1_degRef_2_p_3.0.txt};
		\pgfplotsset{cycle list shift=1}
	\addplot table [x=ndof,y=sqIntegralError] {Data/Exp1/Experiment1_deg_2_degRef_3_p_3.0.txt};
	\addplot table [x=ndof,y=sqLinftyError] {Data/Exp1/Experiment1_deg_2_degRef_3_p_3.0.txt};
	\addplot[dashed,sharp plot,update limits=false] coordinates {(2e1,4e-2) (2e5,4e-4)};
		\addplot[dash dot,sharp plot,update limits=false] coordinates {(2e1,2e-3) (2e5,2e-7)};
	\legend{
	{$\textup{err}_{L^2 \bfV}$},
	{$\textup{err}_{L^\infty L^2}$}};
\end{axis}
\end{tikzpicture}
\caption{Convergence history plots for the experiment in Section \ref{subsec:Slit} with $p=1.5$ (left) and $p=3$ (right) as well as polynomial degrees $r=1$ (solid) and $r=2$ (dotted). The dashed line (\ref{plot:hinv}) indicates rate $\textup{err}_\textup{total} \eqsim \textup{ndof}^{-1/2} \eqsim h$ and the dash-dotted line (\ref{plot:hhalfinv}) the rate $\textup{ndof}^{-1} \eqsim h^2$.}\label{fig:slit1}
\end{center}
\end{figure}
\subsection{Rough in time}\label{sec:RoughInTime}
This experiment solves the $p$-heat equation \eqref{eq:heat} on the
unit square domain $\Omega = (0,1)^2$ in the time interval
$I = (-0.1,0.1)$ with right hand side
$\bff(t) \equiv \textup{sgn}(t) \lvert t\rvert^{-\beta}$ for
fixed parameter $\beta \in \lbrace 0.1,0.5,0.9\rbrace$ and all $t\in I$. The
discrete space $V_h$ is of polynomial degree $r=1$.
The
experiment compares the discrete solutions to a reference solution
$\bfu_\textup{ref}$ that results from a computation with $M = 1024$
time steps and a discrete space $V_h^\textup{ref}$ of higher
polynomial degree $r+1 = 2$ with dimension $\dim V^\textup{ref}_h = 263169$.  Besides
the squared errors from \eqref{eq:DefErrors}, the convergence
history plots in Figure \ref{fig:RoughInTime} display the squared
error
\begin{align*}
  \textup{err}_{L^{p'}\bfS} := \left( \tau \sum_{m=1}^M \lVert \bfS(\nabla \bfu_{h,m}) - \bfS(\nabla \langle \bfu_\textup{ref}\rangle_{J_m})\rVert_{L^{p'}(\Omega)}^{p'}\right)^{2/p'}.
\end{align*}
We expect that~$\bfu$ is smooth in space ($\alpha_x=1$) but rough in
time. In particular,
$\bfu \in C^{1-\beta}(\overline{I}; L^2(\Omega))$. Thus,
Remark~\ref{rem:u-hoelder} would apply with~$\alpha_t =
1-\beta$. However, if we only look at averaged errors,
Theorem~\ref{thm:main} shows that the restriction for $\alpha_t$ comes
from $\bfV(\nabla \bfu) \in N^{\alpha_t,2}(I;L^2(\Omega))$. Inspired
by the regularity of~\cite{CiaMaz20}, we expect that $\partial_t \bfu$
behaves similar to~$\bff$. Hence, heuristic calculations suggest for
$p=1.5$ that $\bfV(\nabla \bfu) \in N^{\alpha_t,2}(I;L^2(\Omega))$
with $\alpha_t = 0.575$ for $\beta= 0.9$ and $\alpha_t = 0.875$ for
$\beta =0.5$ and $\alpha_t = 1$ for $\beta=0.1$. For $p=3$, we expect
$\alpha_t =1$ for $\beta = 0.5,0.1$ and $\alpha_t = 0.65$ for
$\beta = 0.9$.  Indeed, the convergence history plot in Figure
\ref{fig:RoughInTime} indicates rates of convergence far better than
$1-\beta$. For $p=1.5$, the convergence of the error
$\textup{err}_{L^2\bfV}$ agrees with our heuristic calculations and
Theorem \ref{thm:main}. For $p=3$, the rate for $\beta = 0.9$ and
$\beta = 0.5$ seems to be slightly worse than our heuristic
predicts. Either the heuristic is inexact, or we observe a
pre-asymptotic effect. Such a pre-asymptotic effect might also lead to
the slightly worse convergence rate in the experiment of Figure
\ref{fig:singulart} with $p=3$.  The rate of the errors
$\textup{err}_{L^{p'}\bfS} \eqsim \textup{err}_{L^2\bfV}$ is similar
in all computations.  The error $\textup{err}_{L^\infty L^2}$
converges with a better rate than the error $\textup{err}_{L^2\bfV}$.
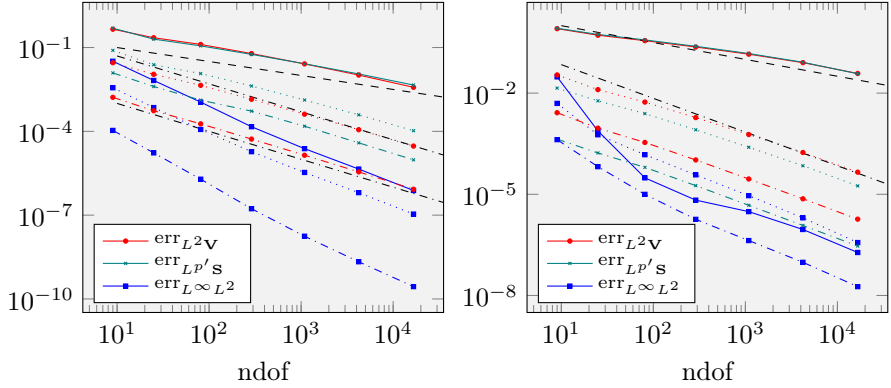
\begin{figure}
\begin{center}
\begin{tikzpicture}
\begin{axis}[
clip=false,
width=.5\textwidth,
height=.45\textwidth,
xmode = log,
ymode = log,
cycle multi list={\nextlist MyColors},
xlabel={ndof},
scale = {1},
clip = true,
legend cell align=left,
legend style={legend columns=1,legend pos= south west,font=\fontsize{7}{5}\selectfont}]
	\addplot table [x=ndof,y=sqVerr] {Data/Exp5/Exp5_avg_p_1.5_beta_-0.9_I_-0.1_0.1.txt};
		\addplot table [x=ndof,y expr=\thisrow{sqAerr}^(2/3)] {Data/Exp5/Exp5_avg_p_1.5_beta_-0.9_I_-0.1_0.1.txt};
		\addplot table [x=ndof,y=sqLinftyError] {Data/Exp5/Exp5_avg_p_1.5_beta_-0.9_I_-0.1_0.1.txt};
	\addplot table [x=ndof,y=sqVerr] {Data/Exp5/Exp5_avg_p_1.5_beta_-0.5_I_-0.1_0.1.txt};
		\addplot table [x=ndof,y expr=\thisrow{sqAerr}^(2/3)] {Data/Exp5/Exp5_avg_p_1.5_beta_-0.5_I_-0.1_0.1.txt};
	\addplot table [x=ndof,y=sqLinftyError] {Data/Exp5/Exp5_avg_p_1.5_beta_-0.5_I_-0.1_0.1.txt};
	\addplot table [x=ndof,y=sqVerr] {Data/Exp5/Exp5_avg_p_1.5_beta_-0.1_I_-0.1_0.1.txt};
		\addplot table [x=ndof,y expr=\thisrow{sqAerr}^(2/3)] {Data/Exp5/Exp5_avg_p_1.5_beta_-0.1_I_-0.1_0.1.txt};
	\addplot table [x=ndof,y=sqLinftyError] {Data/Exp5/Exp5_avg_p_1.5_beta_-0.1_I_-0.1_0.1.txt};
\addplot[dashed,sharp plot,update limits=false] coordinates {(1e1,1e-1) (1e5,1e-3)};
		\addplot[dash dot,sharp plot,update limits=false] coordinates {(1e1,1e-3) (1e5,1e-7)};
		\addplot[dash dot,sharp plot,update limits=false] coordinates {(1e1,5e-2) (1e5,5e-6)};
	\legend{
	{$\textup{err}_{L^2\bfV}$},
	{$\textup{err}_{L^{p'}\bfS}$},
	{$\textup{err}_{L^\infty L^2}$}};
\end{axis}
\end{tikzpicture}
\begin{tikzpicture}
\begin{axis}[
clip=false,
width=.5\textwidth,
height=.45\textwidth,
xmode = log,
ymode = log,
xlabel={ndof},
cycle multi list={\nextlist MyColors},
scale = {1},
clip = true,
legend cell align=left,
legend style={legend columns=1,legend pos= south west,font=\fontsize{7}{5}\selectfont}
]
	\addplot table [x=ndof,y=sqVerr] {Data/Exp5/Exp5_avg_p_3_beta_-0.9_I_-0.1_0.1.txt};
		\addplot table [x=ndof,y expr=\thisrow{sqAerr}^(4/3)] {Data/Exp5/Exp5_avg_p_3_beta_-0.9_I_-0.1_0.1.txt};
	\addplot table [x=ndof,y=sqLinftyError] {Data/Exp5/Exp5_avg_p_3_beta_-0.9_I_-0.1_0.1.txt};
	\addplot table [x=ndof,y=sqVerr] {Data/Exp5/Exp5_avg_p_3_beta_-0.5_I_-0.1_0.1.txt};
		\addplot table [x=ndof,y expr=\thisrow{sqAerr}^(4/3)] {Data/Exp5/Exp5_avg_p_3_beta_-0.5_I_-0.1_0.1.txt};
	\addplot table [x=ndof,y=sqLinftyError] {Data/Exp5/Exp5_avg_p_3_beta_-0.5_I_-0.1_0.1.txt};
	\addplot table [x=ndof,y=sqVerr] {Data/Exp5/Exp5_avg_p_3_beta_-0.1_I_-0.1_0.1.txt};
		\addplot table [x=ndof,y expr=\thisrow{sqAerr}^(4/3)] {Data/Exp5/Exp5_avg_p_3_beta_-0.1_I_-0.1_0.1.txt};
	\addplot table [x=ndof,y=sqLinftyError] {Data/Exp5/Exp5_avg_p_3_beta_-0.1_I_-0.1_0.1.txt};
\addplot[dashed,sharp plot,update limits=false] coordinates {(1e1,1e0) (1e5,1e-2)};\label{plot:Dashed}
		\addplot[dash dot,sharp plot,update limits=false] coordinates {(1e1,7e-2) (1e5,7e-6)};\label{plot:DashDot}
	\legend{
	{$\textup{err}_{L^2\bfV}$},
	{$\textup{err}_{L^{p'}\bfS}$},
	{$\textup{err}_{L^\infty L^2}$},};
\end{axis}
\end{tikzpicture}
\caption{Convergence history plots of the experiment from Section \ref{sec:RoughInTime} with $\beta = 0.9$ (solid), $\beta = 0.5$ (dotted), and $\beta = 0.1$ (dash-dotted) for $p=1.5$ (left) and $p=3$ (right). The dashed line (\ref{plot:Dashed}) indicates the rate $\textup{ndof}^{-1/2} \eqsim h\eqsim \tau$ and the dash-dotted lines (\ref{plot:DashDot}) indicate the rate $\textup{ndof}^{-1} \eqsim h^2 \eqsim \tau^2$.}\label{fig:RoughInTime}
\end{center}
\end{figure}

\subsection{Known solution}\label{sec:KnownSol}
This experiment designs the right-hand side $\bff$ and the inhomogeneous Dirichlet boundary conditions such that the solution reads
\begin{align}\label{eq:SolExp2}
\begin{aligned}
  \bfu(x,t) &= p' |t|^{1/2} |x|^{1/p'}&&(\text{with }1/p + 1/p' = 1),
  \\
  \partial_t \bfu(x,t) & = p'/2\, \textup{sgn}(t) |t|^{-1/2} |x|^{1/p'},
  \\
  \bfV(\nabla \bfu(x,t)) & = |t|^{p/4}|x|^{-1/2} \tfrac{x}{\abs{x}},
  \\
  \bfS(\nabla \bfu(x,t)) & = |t|^{(p-1)/2}|x|^{-1/p'} \tfrac{x}{\abs{x}}.
\end{aligned}
\end{align}  
The time interval $I=(-1,1)$ and the domain is either the centred square domain $\Omega_1 = (-1,1)^2$ or the shifted square domain $\Omega_2 = (1,3) \times (-1,1)$. The function $\bfu$ has singularities in space (at $x=0$) and time (at $t=0$), which cause the reduced regularity $\alpha_t  =1/2$ and $\alpha_x = 1/2$ (for the domain $\Omega_1$) and $\alpha_x=1$ (for $\Omega_2$). We include the inhomogeneous Dirichlet boundary conditions by averaging the nodal interpolation: Let $\Pi_h:W^{1,p}\to V_h$ be the nodal interpolation operator onto the space of piece-wise affine polynomials $V_h$ from \eqref{def:Vh}, then we set at the boundary 
\begin{align*}
\bfu_{m,h}|_{\partial\Omega} := \langle \Pi_h \bfu \rangle_{J_m}|_{\partial \Omega}\qquad\text{for all }m=1,\dots,M.
\end{align*}
The convergence history plots in Figure \ref{fig:singulart} display the squared errors
\begin{align*}
\textup{err}_{L^\infty L^2} & :=\max_{1\leq m \leq M} \lVert \bfu_{h,m} - \langle \bfu_\textup{ref}\rangle_{J_m}\rVert_{L^2(\Omega)}^2,\\
\textup{err}_{L^2 \langle \bfV\rangle_{J_m}} & := \tau \sum_{m=1}^M \lVert \bfV (\nabla \bfu_{h,m}) - \langle \bfV(\nabla \bfu) \rangle_{J_m}\rVert_{L^2(\Omega)}^2,\\
\textup{err}_{L^2 \bfV} & := \sum_{m=1}^M \int_{J_m} \lVert \bfV (\nabla \bfu_{h,m}) -\bfV(\nabla \bfu)\rVert_{L^2(\Omega)}^2\,\mathrm{d}\sigma\\
&\hphantom{:} = \textup{err}_{L^2 \langle \bfV\rangle_{J_m}} +  \sum_{m=1}^M \int_{J_m} \lVert  \bfV (\nabla \bfu) - \langle \bfV(\nabla \bfu) \rangle_{J_m}\rVert_{L^2(\Omega)}^2\,\mathrm{d}\sigma,\\
\textup{err}_{L^{p'}\langle \bfS\rangle_{J_m}} & := \left(\tau \sum_{m=1}^M \lVert \bfS (\nabla \bfu_{h,m}) - \langle \bfS(\nabla \bfu) \rangle_{J_m}\rVert^{p'}_{L^{p'}(\Omega)}\right)^{2/p'}.
\end{align*}
As in the previous experiments, the error $\textup{err}_{L^2 \bfV}$
dominates the error $\textup{err}_{L^\infty L^2}$. The solution to the
problem on $\Omega_1$ (solid lines) converges with the expected rate
$h \eqsim \tau$ for $p=1.5$. Since
$\bfS(\nabla u) \in L^{p'}(I; N^{\frac 13,p'}(\Omega))$, the rate of
$\textup{err}_{L^{p'}\langle \bfS\rangle_{J_m}}$ is worse but agrees
with the possible approximation rate.
For $p=3$ and $\Omega_1$ the rate of the error $\textup{err}_{L^2\bfV}$ is slightly worse than $\tau$.
This might be some pre-asymptotic effect or it is caused by quadrature errors due to the highly singular right-hand side $\bff := \partial_t\bfu - \textup{div}\,\bfS(\nabla \bfu)$. 
The averaged errors in the computation on $\Omega_2$ overcome, as stated in Theorem \ref{thm:main}, the reduced regularity of $\bfu$ in $t=0$: the squared error $\textup{err}_{L^\infty L^2}+\textup{err}_{L^2\bfV}$ converges with the optimal rate $h^2\eqsim \tau^2$. As in the previous experiments, the error $\textup{err}_{L^\infty L^2}$ is much smaller than the error $\textup{err}_{L^2\bfV}$. In addition, the error $\textup{err}_{L^\infty L^2}$ converges faster than $\textup{err}_{L^\infty\bfV}$ for $p=1.5$ and $\Omega_1$. For $p=3$ and $\Omega_1$, the rates of $\textup{err}_{L^\infty L^2}$ and $\textup{err}_{L^\infty\bfV}$ are similar.
\begin{figure}
\begin{center}
\begin{tikzpicture}
\begin{axis}[
clip=false,
width=.5\textwidth,
height=.45\textwidth,
xmode = log,
ymode = log,
xlabel={ndof},
cycle multi list={\nextlist MyColors4},
scale = {1},
clip = true,
legend cell align=left,
legend style={legend columns=1,legend pos= south west,font=\fontsize{7}{5}\selectfont}
]

	\addplot table [x=ndof,y=sqVerr] {Data/Exp4/Experiment4_avg_p_1.5_shifted_False_I_-1_1.txt};
	\addplot table [x=ndof,y=sqVerr1] {Data/Exp4/Experiment4_avg_p_1.5_shifted_False_I_-1_1.txt};		
		\addplot table [x=ndof,y=sqLinftyError] {Data/Exp4/Experiment4_avg_p_1.5_shifted_False_I_-1_1.txt};
	\addplot table [x=ndof,y=sqAerr] {Data/Exp4/Experiment4_avg_p_1.5_shifted_False_I_-1_1.txt};
	\addplot table [x=ndof,y=sqVerr] {Data/Exp4/Experiment4_avg_p_1.5_shifted_True_I_-1_1.txt};
	\addplot table [x=ndof,y=sqVerr1] {Data/Exp4/Experiment4_avg_p_1.5_shifted_True_I_-1_1.txt};
			\addplot table [x=ndof,y=sqLinftyError] {Data/Exp4/Experiment4_avg_p_1.5_shifted_True_I_-1_1.txt};
		\addplot table [x=ndof,y=sqAerr] {Data/Exp4/Experiment4_avg_p_1.5_shifted_True_I_-1_1.txt};
	\addplot[dashed,sharp plot,update limits=false] coordinates {(3e1,7e-1) (3e5,7e-3)};\label{plot:hh2}
		\addplot[dash dot,sharp plot,update limits=false] coordinates {(3e1,1e-2) (3e5,1e-6)};\label{plot:h2}
		\addplot[dotted,sharp plot,update limits=false] coordinates {(3e1,7e-1) (3e5,7*pow(10,-7/3)};\label{plot:dotted2div3}
	\legend{
	{$\textup{err}_{L^2 \bfV}$},
	{$\textup{err}_{L^2 \langle \bfV\rangle_{J_m}}$},
	{$\textup{err}_{L^\infty L^2}$ },
	{$\textup{err}_{L^{p'} \langle\bfS\rangle_{J_m}}$}};
\end{axis}
\end{tikzpicture}
\begin{tikzpicture}
\begin{axis}[
clip=false,
width=.5\textwidth,
height=.45\textwidth,
xmode = log,
ymode = log,
xlabel={ndof},
cycle multi list={\nextlist MyColors4},
scale = {1},
clip = true,
legend cell align=left,
legend style={legend columns=1,legend pos= south west,font=\fontsize{7}{5}\selectfont}
]
	\addplot table [x=ndof,y=sqVerr] {Data/Exp4/Experiment4_avg_p_3_shifted_False_I_-1_1.txt};
	\addplot table [x=ndof,y=sqVerr1] {Data/Exp4/Experiment4_avg_p_3_shifted_False_I_-1_1.txt};
	\addplot table [x=ndof,y=sqLinftyError] {Data/Exp4/Experiment4_avg_p_3_shifted_False_I_-1_1.txt};	
	\addplot table [x=ndof,y=sqAerr] {Data/Exp4/Experiment4_avg_p_3_shifted_False_I_-1_1.txt};
	\addplot table [x=ndof,y=sqVerr] {Data/Exp4/Experiment4_avg_p_3_shifted_True_I_-1_1.txt};
	\addplot table [x=ndof,y=sqVerr1] {Data/Exp4/Experiment4_avg_p_3_shifted_True_I_-1_1.txt};	
		\addplot table [x=ndof,y=sqLinftyError] {Data/Exp4/Experiment4_avg_p_3_shifted_True_I_-1_1.txt};
		\addplot table [x=ndof,y=sqAerr] {Data/Exp4/Experiment4_avg_p_3_shifted_True_I_-1_1.txt};
	\addplot[dashed,sharp plot,update limits=false] coordinates {(3e1,7e-1) (3e5,7e-3)};
		\addplot[dash dot,sharp plot,update limits=false] coordinates {(3e1,1e-2) (3e5,1e-6)};
	\legend{
	{$\textup{err}_{L^2 \bfV}$},
	{$\textup{err}_{L^2 \langle \bfV\rangle_{J_m}}$},
	{$\textup{err}_{L^\infty L^2}$},
	{$\textup{err}_{L^{p'} \langle\bfS\rangle_{J_m}}$}};
\end{axis}
\end{tikzpicture}
\caption{Convergence history plots for the experiment in Section \ref{sec:KnownSol} with $p=1.5$ (left), $p=3$ (right) and domains $\Omega_1$ (solid), $\Omega_2$ (dotted). The dotted line (\ref{plot:dotted2div3}) indicates the rate $\textup{ndof}^{-1/3}\eqsim h^{2/3}$, the dashed line (\ref{plot:hh2}) indicates the rate $\textup{ndof}^{-1/2}\eqsim h$ and the dash-dotted line (\ref{plot:h2}) indicates the rate $\textup{ndof}^{-1} \eqsim h^2$.}\label{fig:singulart}
\end{center}
\end{figure}
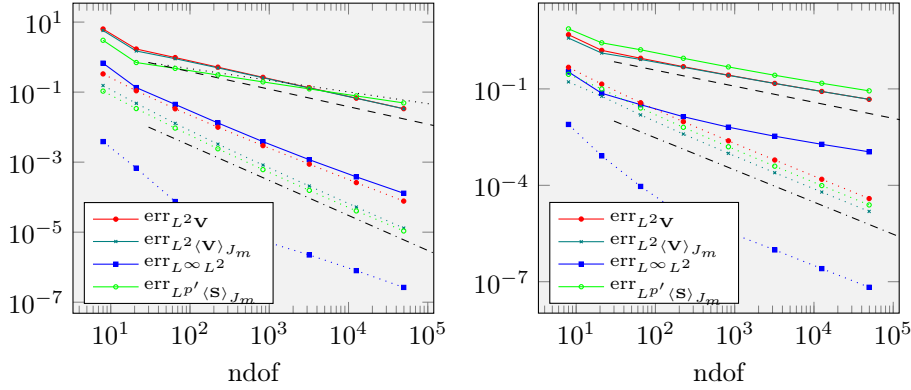

\appendix

\section{Orlicz spaces}
\label{sec:Orlicz spaces}

The following definitions and results are standard in the theory of
Orlicz spaces.  A continuous, convex and strictly increasing function
$\phi\,:\, [0,\infty) \to [0,\infty)$ satisfying
\begin{align*}
  \lim_{t\rightarrow0}\frac{\phi(t)}{t}=
  \lim_{t\rightarrow\infty}\frac{t}{\phi(t)}=0
\end{align*}
is called an $N$-function.

We say that $\phi$
satisfies the $\Delta_2$--condition, if there exists $c > 0$ such that
for all $t \geq 0$ holds $\phi(2t) \leq c\, \phi(t)$. By
$\Delta_2(\phi)$ we denote the smallest such constant. Since $\phi(t)
\leq \phi(2t)$ the $\Delta_2$-condition is equivalent to $\phi(2t)
\eqsim \phi(t)$ uniformly in $t$. Note that if $\Delta_2(\phi) < \infty$ then $\phi(t) \eqsim \phi(c\,t)$
uniformly in $t\geq 0$ for any fixed $c>0$. For a family $\phi_\lambda$ of $N$-functions we define
$\Delta_2(\set{\phi_\lambda}) := \sup_\lambda \Delta_2(\phi_\lambda)$.
By $L^\phi$ and $W^{k,\phi}$, $k\in \setN_0$,  we denote the classical Orlicz and
Orlicz-Sobolev spaces, i.e.\ $f \in L^\phi$ iff $\int
\phi(\abs{f})\dx < \infty$ and $f \in W^{k,\phi}$ iff $ \nabla^j f
\in L^\phi$, $0\le j\le k$.

By $\phi^*$ we denote the conjugate N-function of $\phi$, which is
given by $\phi^*(t) = \sup_{s \geq 0} (st - \phi(s))$. Then $\phi^{**}
= \phi$.

The following definitions and results are summarized
from~\cite{DR,BelDieKre2012,DieForTomWan20}. 
\begin{definition}
  \label{ass:phipp}
  Let $\phi$ be an N-function. 
  We say that $\phi$ is \emph{uniformly convex}, if
  $\phi$ is $C^1$ on $[0,\infty)$ and $C^2$ on $(0,\infty)$ and
  assume that 
  \begin{align}
    \label{eq:phipp}
    \phi'(t) &\eqsim t\,\phi''(t)
  \end{align}
  uniformly in $t > 0$. The constants hidden in $\eqsim$ are called the
  \emph{characteristics of~$\phi$}.
\end{definition}
Note that~\eqref{eq:phipp} is stronger than
$\Delta_2(\phi,\phi^*)<\infty$. In fact, the $\Delta_2$-constants can
be estimated in terms of the characteristics of~$\phi$.

Associated to an uniformly convex $N$-function $\phi$ we define the tensors
\begin{align*}
  \bfS(\bfxi)&:=\frac{\phi'(\abs{\bfxi})}{\abs{\bfxi}}\bfxi,\quad
                    \bfxi\in\mathbb R^{N\times n}
  \\
  \bfV(\bfxi)&:=\sqrt{\frac{\phi'(\abs{\bfxi})}{\abs{\bfxi}}}\,\bfxi,\quad
                    \bfxi\in\mathbb R^{N\times n}.
\end{align*}
We define the \emph{shifted} $N$-function $\phi_a$ for $a\geq 0$ by
\begin{align}
  \label{eq:def_shift}
  \phi_a(t) &:= \int_0^t \frac{\phi'(a+s)}{a+s} s \ds.
\end{align}
In our application (cf. \eqref{eq:def-S}, \eqref{eq:def-V} and \eqref{eq:def-phi}) $\phi$ is given by
\begin{align*}
\phi(t) = \int_0^t (\kappa +  s)^{p-2} s \dd s
\end{align*}
and the tensors are
\begin{align*}
\bfS(\bfxi) = \left( \kappa + \abs{\bfxi} \right)^{p-2} \bfxi \quad \text{ and } \quad \bfV(\bfxi) = \left( \kappa + \abs{\bfxi} \right)^\frac{p-2}{2} \bfxi.
\end{align*}
\begin{lemma}[Equivalence lemma]
  \label{lem:hammer}
  We have
  \begin{align*}
    \begin{aligned}
      \big({\bfS}(\bfP) - {\bfS}(\bfQ)\big) \cdot
      \big(\bfP-\bfQ \big) &\eqsim \bigabs{ \bfV(\bfP) -
        \bfV(\bfQ)}^2
      \\
      &\eqsim \phi_{\abs{\bfP}}(\abs{\bfP - \bfQ})
      \\
      &\eqsim \phi''\big( \abs{\bfP} + \abs{\bfQ}
      \big)\abs{\bfP - \bfQ}^2
    \end{aligned}
  \end{align*}
  uniformly in $\bfP, \bfQ \in \setR^{N \times n}$.  Moreover,
  uniformly in $\bfQ \in \setR^{N \times n}$,
  \begin{align*}
    \bfS(\bfQ) \cdot \bfQ &\eqsim \abs{\bfV(\bfQ)}^2\eqsim
                                 \phi(\abs{\bfQ})\\
    \abs{{\bfS}(\bfP) - {\bfS}(\bfQ)}&\eqsim\big(\phi_{\abs{\bfP}}\big)'(\abs{\bfP - \bfQ}).
  \end{align*}
  The constants depend only on the characteristics of $\phi$.
\end{lemma}
\begin{lemma}
  \label{lem:young2}
  Let $\phi$ be an uniformly convex N-function. Then for each $\delta>0$ there
  exists $C_\delta \geq 1$ (only depending on~$\delta$ and the
  characteristics of~$\phi$) such that
  \begin{align*}
    \big({\bfS}(\bfP) - {\bfS}(\bfQ)\big) \cdot
    \big(\bfR-\bfQ \big) &\leq \delta\bigabs{ \bfV(\bfP) -
                           \bfV(\bfQ)}^2+C_\delta \bigabs{ \bfV(\bfR)                          -\bfV(\bfQ)}^2
  \end{align*}
  for all $\bfP,\bfQ,\bfR\in\setR^{N\times n}$.
\end{lemma}

\begin{lemma}[Change of Shift]
 \label{lem:shift_ch}
 Let $\phi$ be an uniformly convex N-function. Then for each $\delta>0$ there
 exists $C_\delta \geq 1$ (only depending on~$\delta$ and the
 characteristics of~$\phi$) such that
  \begin{align*}
    \phi_{\abs{\bfa}}(t)&\leq C_\delta\, \phi_{\abs{\bfb}}(t)
    +\delta\, \abs{\bfV(\bfa) - \bfV(\bfb)}^2,
  \end{align*}
 for all $\bfa,\bfb\in\setR^{N \times n}$ and $t\geq0$.
\end{lemma}

\newcommand{\etalchar}[1]{$^{#1}$}
\providecommand{\bysame}{\leavevmode\hbox to3em{\hrulefill}\thinspace}
\providecommand{\MR}{\relax\ifhmode\unskip\space\fi MR }
\providecommand{\MRhref}[2]{%
  \href{http://www.ams.org/mathscinet-getitem?mr=#1}{#2}
}
\providecommand{\href}[2]{#2}


\end{document}